\documentclass[12pt,absolute]{mymathart}
\usepackage[theorems]{mymathmacros}
\usepackage{ogonek,amscd,graphicx}
\usepackage[all,cmtip]{xy}

\usepackage{color}

\usepackage{hyperref}
\numberwithin{equation}{section}
\newcommand{\raequiv}{\simeq}
\newcommand{\dom}{\operatorname{dom}}
\DeclareMathOperator{\Mob}{M\ddot{o}b}
\newcommand{\MobR}{\Mob_{\R}}

\newcommand{\HomeoR}{\operatorname{Homeo}_{\R}}
\newcommand{\id}{\operatorname{id}}
\newcommand{\MfR}{M_{f}^{\R}}
\newcommand{\MfRt}{\widetilde{\MfR}}
\newcommand{\MfC}{M_{f}^{S^1}}
\newcommand{\MfCt}{\widetilde{M_{f}^{S^1}}}
\newcommand{\BR}{\B_{\operatorname{real}}}
\newcommand{\CV}{\operatorname{CV}}
\newcommand{\Part}{\operatorname{Part}}

\newcommand{\dens}{\operatorname{dens}}
\newcommand{\meas}{\operatorname{meas}}

\newcommand{\CS}{\C^*}

\newcommand{\SR}{\mathcal{S}_{\R}}
\newcommand{\SC}{\mathcal{S}_{S^1}}

\newcommand{\Crit}{\operatorname{Crit}}

\newcommand{\s}{\underline{s}}
\newcommand{\B}{\mathcal{B}}
\newcommand{\IR}{I_{\R}}
\newcommand{\HH}{\mathbb H}

\renewcommand{\P}{\mathcal{P}}
\renewcommand{\S}{\mathcal{S}}

\title[Density of Hyperbolicity for Transcendental Maps]{%
Density of hyperbolicity
for classes \\ of real 
transcendental entire functions \\ and circle maps}
\author{Lasse Rempe-Gillen}
\address{Dept.\ of Math.\ Sciences, University of Liverpool, Liverpool L69 7ZL, United Kingdom}
\email{l.rempe@liverpool.ac.uk}
\author{Sebastian van Strien}
\address{Sebastian van Strien,  Department of Mathematics, Imperial College, 180 Queen's Gate, London SW7 2AZ, UK.}
\email{s.van-strien@imperial.ac.uk}
\thanks{The first author was supported by EPSRC Grant EP/E017886/1, Fellowship EP/E052851/1 and a Philip Leverhulme Prize.}
\date{\today}


\begin{document}

\begin{abstract} 
We prove density of hyperbolicity in spaces of 
(i)  real transcendental entire functions, bounded on the real line,
  whose singular set is 
 finite and real and 
 (ii) transcendental functions 
 $f\colon \C\setminus \{0\}\to \C\setminus \{0\}$ that preserve the circle
 and whose singular set (apart from $0,\infty$) is finite and contained in the circle.
 In particular, we prove density of hyperbolicity in the famous Arnol'd family
 of circle maps and its generalizations, and solve a number of other
 open problems for these functions, including three conjectures of
 de Melo, Salom\~ao and Vargas \cite{melo-etal}. 
 
We also prove density of (real) 
 hyperbolicity for certain families as in (i) but
 without the boundedness condition. Our results apply, in particular,
 when the functions
 in question have only finitely many critical points and asymptotic
 singularities, or when there are no asymptotic values and the degree of critical
   points is uniformly bounded.
\end{abstract}

\maketitle

\section{Introduction}

 Among dynamical systems, those that are \emph{hyperbolic}, a property also called \emph{Axiom A}, have particularly simple behaviour and are the easiest to understand (for a definition in our context see below).    For this reason, \emph{density of hyperbolicity}\ -- the question whether any system in a given parameter space
   can be perturbed to a hyperbolic one\ -- is one of the central problems of one-dimensional dynamics.   
 (It has been known for about 50 years that the answer is negative in higher dimensions;  
 for references and recent results, see for example \cite{MR1779562} and \cite{MR2818683}.)

Recently, there has been major progress on this problem in the real setting. 
Lyubich \cite{Lyubich_dynI+II} and, independently,
Graczyk and \'Swi\k{a}tek \cite{MR1469316} solved the problem for the
real quadratic family $x\mapsto x^2 + c$, while it was solved 
by Kozlovski, Shen and the second author for
real polynomials with real critical points 
in \cite{KSS1} and for general interval maps and circle maps in  
\cite{densityaxiomadimensionone}.
For a discussion of related results, see \cite{strien-survey}.

 As an example of questions that are left open by these theorems, 
  let us consider the most famous family of circle maps: the 
  \emph{Arnol'd family}
\[ F_{\mu_1,\mu_2}(t)= t+ \mu_1 + \mu_{2}\sin(2 \pi t); \quad \mu_1\in\R, \mu_2>0. \]
This family describes the behaviour of a periodically forced nonlinear oscillator, and has been
 used to model a variety of physical and biological systems.
    
 It is well-known that hyperbolicity is dense in the region
 where the map is a circle diffeomorphism, i.e.\ for $\mu_2<1/(2\pi)$. 
 In the non-invertible case, $\mu_2>1/(2\pi)$, 
 \cite[Theorem 2]{densityaxiomadimensionone} implies that
 $F_{\mu_1,\mu_2}(t)$ can be perturbed to a hyperbolic circle map, 
 and indeed to
 a hyperbolic trigonometric polynomial of high degree. However, we would 
 like the perturbation 
       to remain within
  the same family; that is, we ask whether the set of
 parameters   $(\mu_1,\mu_2)$  for which both critical points belong to the basins of periodic
 attractors is dense in the region $\mu_2>1/(2\pi)$ (Figure \ref{fig:arnold}). This question, and in fact even
  the density of \emph{structural stability} (see below), had remained open prior
  to our work. 

\begin{figure}
\includegraphics[width=\textwidth]{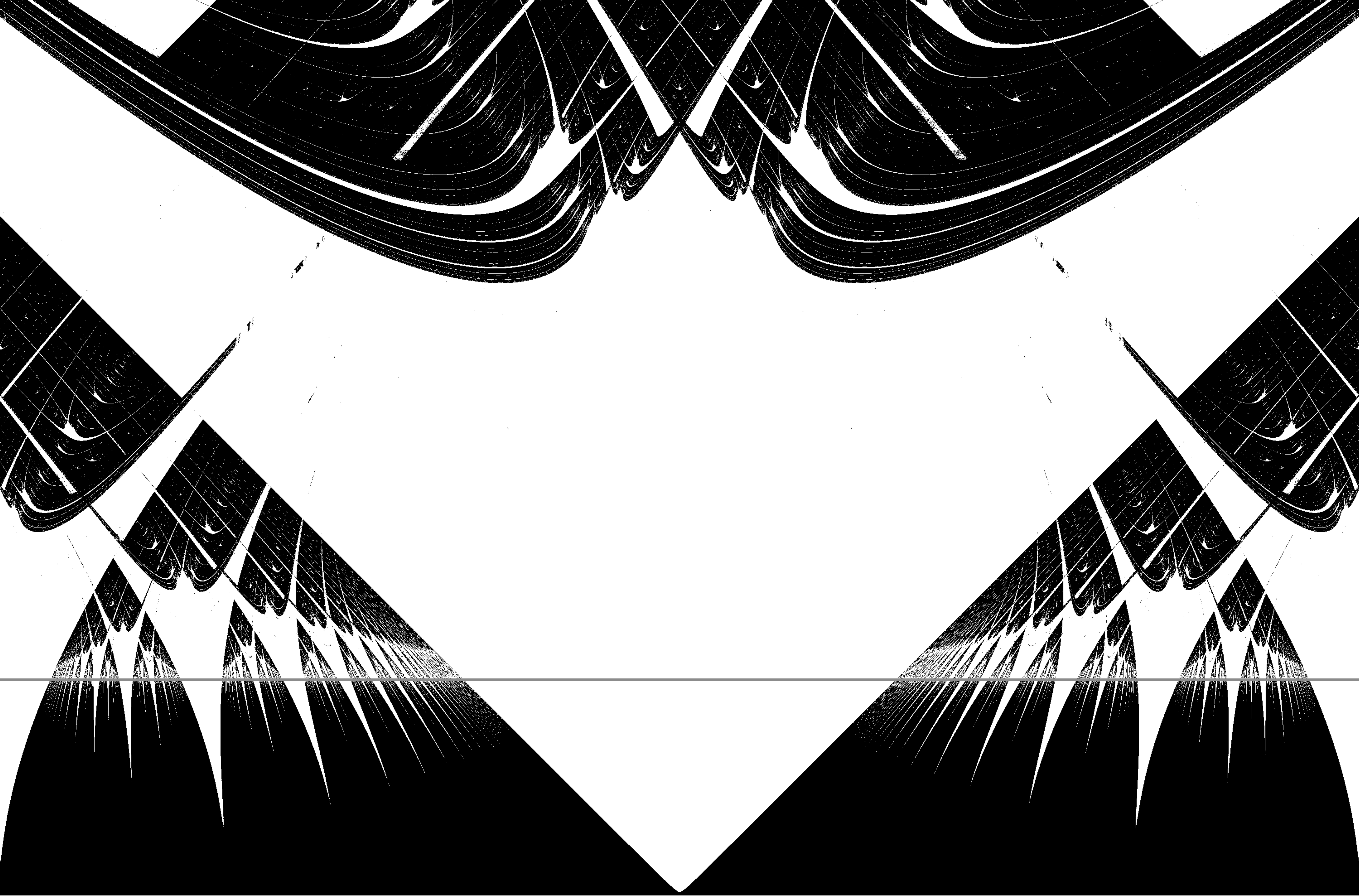}
\caption{Parameter space for the Arnol'd family in the region $(\mu_1,\mu_2)\in
   (-1/2,1/2)\times (0,1/2+1/(2\pi))$. \label{fig:arnold}White regions correspond to points where a numerical experiment indicates 
  that both
  critical points belong to attracting basins. 
  The critical line $\mu_2=1/(2\pi)$ is indicated in grey; note that the Arnol'd tongues
  in the invertible region lead to hyperbolic components above the critical line, but that
 other hyperbolic regions exist also.}
\end{figure}

As a further example, we discuss the families of \emph{real cosine maps}
  and \emph{degenerate standard maps}:
  \begin{align*}
     C_{a,b}(x) &:= a\sin(x) + b\cos(x), \quad 
     (a,b)\in\R^2\setminus\{(0,0)\}, \\
     S_{a,b}(x) &:= ax e^{x} + b, \quad
              a\in\R\setminus\{0\}, b\in\R. 
  \end{align*}
 These are natural families of transcendental entire functions. (The study of
  transcendental dynamics, which goes back to Fatou, has received increasing
  attention recently, partly due to the discovery of deep connections with
  polynomial and rational dynamics. We refer e.g.\ to \cite{strahlen} for
  some examples.) 
 Again, \cite{densityaxiomadimensionone} implies that each of the functions above
  can be approximated by a hyperbolic \emph{polynomial} of high degree, but
  we should look for hyperbolic perturbations
  within the families; see
  Figure \ref{fig:cosine}.  
 (In the case of maps such as $S_{a,b}$, which are unbounded on the real axis,
  we shall need to take care to use the right notion of
  hyperbolicity; see Definition 
   \ref{defn:realhyperbolic} below.)

\begin{figure}
\includegraphics[width=\textwidth]{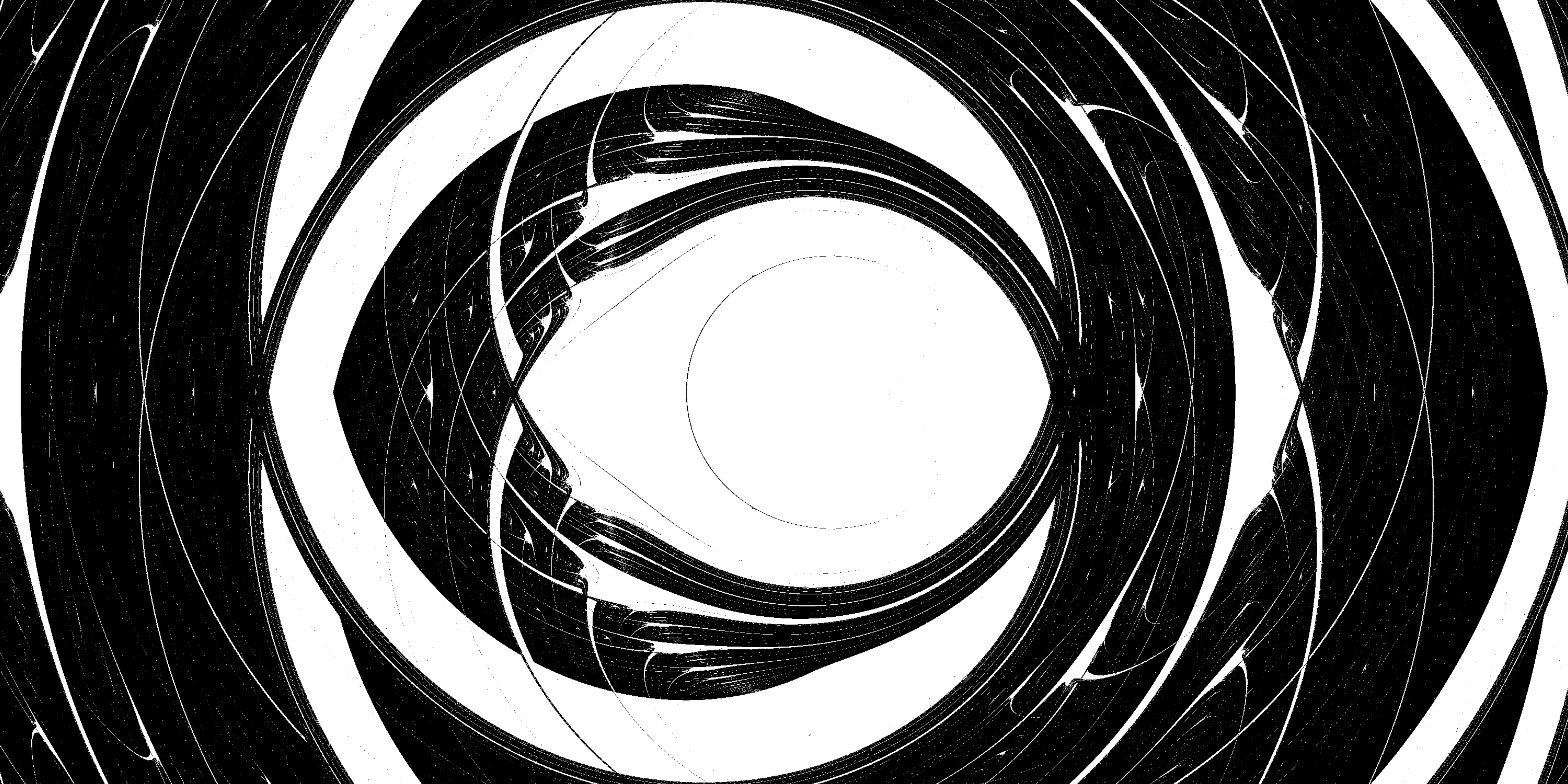}  
\caption{The real cosine family\label{fig:cosine}. As in Figure \ref{fig:arnold}, white regions correspond to hyperbolic maps.}
\end{figure}

 We give positive answers to
  all of the above questions and in fact establish density of
  hyperbolicity 
for  a large general class of parameter spaces of transcendental
   entire functions and circle maps. 
  To do so, we must abandon the proof strategy of
  \cite{densityaxiomadimensionone}, which relies heavily on the use of
  \emph{polynomial-like mappings}. Instead, we return to the methods
  of \cite{KSS1}. The difficulty here is that we require global 
  rigidity statements, in particular regarding the
  \emph{absence of invariant line fields} on the complex plane, which 
  become more difficult to establish given that our functions are
  transcendental.  In order to prove our results, we shall need to combine and
  adapt
  a number of ingredients:
  \begin{enumerate}
   \item rigidity results for maps of the interval and circle maps; 
   \item rigidity results for the dynamics of transcendental entire functions near infinity;
   \item an argument to establish the absence of invariant line fields
     on certain subsets of the complex plane; 
   \item the function-theoretic construction of natural parameter spaces.
  \end{enumerate}
As far as we know, 
density of hyperbolicity
had not previously 
been established in any nontrivial family of transcendental 
functions. 
 
\subsection*{Statement of results for real transcendental functions}

If $f:\C\to\C$ is a transcendental entire function, we denote by
 $S(f)$ the \emph{set of (finite) singular values} of $f$. 
 That is, $S(f)\subset\C$ is the smallest closed set such that
   \[ f: f^{-1}(\C\setminus S(f))\to \C\setminus S(f) \]
 is an unbranched covering. 

 Let $f$ belong to the \emph{Speiser class} $\S$ of transcendental entire
  functions for which $S(f)$ is finite. We say that 
  $f$ is \emph{hyperbolic}
   if every singular value belongs to a basin of attraction; as in the rational
  case, this definition implies uniform expansion on the
  Julia set (see \cite[Theorem C]{ripponstallardhyperbolic} or \cite[Lemma 5.1]{boettcher}).
  When studying density of hyperbolicity, it is reasonable to restrict
  to the class $\S$.  Indeed, for maps with infinite sets
 of singular values, the associated natural parameter spaces
 will be infinite-dimensional, the number of periodic attractors may
 become infinite, there might exist wandering domains, and even
 density of structural stability may fail. In fact, it is not entirely
 clear whether ``hyperbolicity'' is a notion that makes sense when
 the set $S(f)$ is unbounded.

Since we are interested in real dynamics, we 
consider only
\emph{real} transcendental entire functions; i.e.\ those
that satisfy $f(\R)\subset\R$. Furthermore, we assume that
all singular values are also real; i.e. we study the class 
 \[ \SR := \{f:\C\to\C \text{ real transcendental entire}:
                   S(f)\text{ is finite and contained in $\R$}\}. \]
This is a reasonable restriction if our goal is to study 
 hyperbolicity in the complex sense. It seems 
 sensible to expect that density of hyperbolicity
 \emph{on the real line} also holds without the assumption that
 $S(f)\subset\R$, but our current methods will not yield this. 
 We note that a function $f\in\SR$ may have 
 non-real critical \emph{points}, but only real critical \emph{values}. 

To study density of hyperbolicity we must first clarify what perturbations
we allow. It is natural to require these to preserve the global
properties of the original map: for example, if a function is bounded on the real 
line, the approximating map should have the same property. 
It turns out that the correct notion is to seek perturbations of a map $f$
that are entire functions of the form $\psi\circ f \circ \phi^{-1}$, where  $\psi$ and $\phi$ belong
to the class $\HomeoR$ of orientation-preserving homeomorphisms of the complex plane that commute with complex conjugation 
 and restrict to order-preserving homeomorphisms of the real line. 
Our first result concerns
maps $f\in\SR$ that are bounded on the real axis.

 \begin{thm}[Perturbation of bounded functions] \label{thm:perturbbounded}
  Suppose that $f\in\SR$ is bounded on the real axis. Then there exist
   $\phi,\psi\in\HomeoR$ arbitrarily close to the identity such that
   $g := \psi\circ f \circ \phi^{-1}$ is entire and hyperbolic. 
 \end{thm}
 
Another (more practical) point of view is to study perturbations that belong to natural families 
of functions in $\SR$. Using Theorem \ref{thm:perturbbounded}, we can deduce the following
result in this spirit. Let $\MobR\subset\HomeoR$ denote the set of all affine maps $M(z)=az+b$, $a>0$, 
$b\in\R$. 

\begin{thm}[Density of hyperbolicity in families of bounded functions] \label{thm:densitybounded}
 Let $n\geq 1$ and let $N$ be an $n$-dimensional (topological)
  manifold. Suppose that
  $(f_{\lambda})_{\lambda\in N}$ is a continuous family of functions
  $f_{\lambda}\in \SR$ such that
 \begin{enumerate} 
  \item $f_{\lambda}|_\R$ is bounded for all $\lambda\in N$;
  \item $\# S(f_{\lambda})\leq n$ for all $\lambda\in N$;\label{item:numbersingular}
  \item no two maps $f_{\lambda_1}$ and $f_{\lambda_2}$ are conjugate
   by a map $M\in \MobR$. \label{item:nontrivial}
 \end{enumerate}
 Then the set $\{\lambda\in N: f_{\lambda}\text{ is hyperbolic}\}$ is 
  open and dense in $N$. 
\end{thm}
  Assumption (\ref{item:numbersingular}) is needed: as in \cite{strien-robustchaos}
  it is not hard to construct  $d$-parameter families with $d< n$ so that
  {\em no} map within this family is hyperbolic.

 We note that it is possible to embed every $f\in\SR$ with
   $\# S(f)=n$ in an
  $n$-dimensional family 
  $f_{\lambda}$ satisfying (\ref{item:numbersingular}) and (\ref{item:nontrivial}) in a natural fashion
  (see Section \ref{sec:parameter}). Furthermore, if $f|_{\R}$ is bounded, then all
  elements of this family will also be bounded. 
    
As a particular case, the above theorems imply density of hyperbolicity in the 
real cosine family mentioned above. It also holds for general
real trigonometric polynomials for which all critical values are real.
  (See also Corollary~\ref{cor:trigopol} below for a more general statement regarding
   circle maps.)

\begin{cor}[Density of hyperbolicity for trigonometric polynomials]\label{cor:trigopolreal}
The set of parameters $(a,b)$ for which the cosine map $C_{a,b}$ is hyperbolic forms
an open and dense subset of $\R^2$.

More generally, let $n\geq 1$. Then hyperbolicity is dense in the space
 of real trigonometric polynomials
   \begin{equation} \label{eqn:trig}
 f(x) = a_0 + \sum_{j=1}^n\bigl(a_j \cos(jx)+b_j\sin(jx)\bigr) 
  \end{equation}
 for which all critical values are real.  
\end{cor} 
\begin{proof}
All functions $C_{a,b}$ belong to the class $\SR$, with exactly
two critical values and no asymptotic values. Furthermore, no two different
maps $C_{a,b}$ are conjugate by a M\"obius transformation
  $z\mapsto \alpha z + \beta$, $\alpha>0$, $\beta\in\R$
  (Lemma \ref{lem:cos}). 

We note that if $f$ is a trigonometric polynomial and 
 $g=\psi\circ f \circ \phi^{-1}$ is entire with $\psi$ and $\phi$
 close to the identity, then $g$ is conformally conjugate to a
 trigonometric polynomial of the same degree whose coefficients 
 are close to those of $f$ (Lemma \ref{lem:trig}). 

Hence the corollary indeed follows from Theorems \ref{thm:densitybounded} and 
    \ref{thm:perturbbounded}.
\end{proof}

For functions that are unbounded along the real axis, such as the family
  $S_{a,b}$, we need to relax
our notion of hyperbolicity somewhat. The reason is that here
 some singular values may ``escape to infinity'' (i.e., converge
to infinity under iteration). In this case, the function is not 
hyperbolic
in the complex sense, as $\infty$ is not a hyperbolic attractor.
However, such a singular value cannot be perturbed into an attracting
basin by a \emph{real} perturbation. For example, consider the
real exponential family, $E_{a}(x) = \exp(x)+ a$, $a \in\R$.
For $a<1$, $E_a(z)$ is hyperbolic, but for $a>1$, the singular value
$a$, and indeed  every real starting value $x$, converges to $\infty$
under iteration. These maps are not hyperbolic in the
complex plane~-- indeed, the Julia set is the whole complex plane, 
the maps are far from uniformly expanding,
and their topological dynamics 
is still not completely understood~-- but it seems reasonable to describe their action on $\R$
as hyperbolic, motivating the following definition.

\begin{defn}[Real-hyperbolicity of maps in $\SR$]\label{defn:realhyperbolic}
A function $f\in \SR$ is called \emph{real-hyperbolic} if every
 singular value either belongs to a basin of attraction or tends to infinity under
 iteration.  
\end{defn}

When $f|_{\R}$ is  bounded, this corresponds
 to the usual definition of hyperbolicity. 
We should note that, if $f_{\lambda}$ is a family
 of functions in $\SR$ for which the number of singular
 values is constant, then any real-hyperbolic parameter
 $\lambda_0$ for which there are {\em no critical relations}
 is {\em real-structurally stable} within the family. By this we mean that
 any nearby map $f_{\lambda}$ is conjugate
 to $f_{\lambda_0}$ on the real line. (However, they are not necessarily conjugate
 in the complex plane; indeed $\exp(x)+a$ and $\exp(x)+b$ are not
topologically  conjugate for $a>b>-1$, see \cite{douadygoldberg}.)
  Here we say that $f$ has no critical relations if
  no critical point or asymptotic value of $f$ is eventually mapped
 onto a critical point.
 Indeed, real-structural stability follows from the fact that $f_{\lambda_0}$ and a nearby map
 will be combinatorially and hence topologically
 conjugate on the real line (see Lemma \ref{lem:comb-conj}).

 \medskip

It is reasonable to conjecture that real-hyperbolicity is dense 
in every full parameter space in $\SR$. In this paper, we establish this
conjecture for functions satisfying quite general additional ``geometrical'' properties, 
  by which we mean that these properties depend on the function-theoretic behaviour of the maps
  (rather than their dynamics). There are two such conditions, each of which will allow us 
  to establish density of real-hyerbolicity. The first of these relates to the 
  geometry of the finite singular values of $f$: We shall say that a function $f\in\S$ has
  \emph{bounded criticality} if $f$ has no asymptotic values and the degree of critical points of
   $f$ is uniformly bounded. This class of entire functions appears in 
   work of Mihaljevi\'c-Brandt \cite{helenaorbifolds}, which shows that these maps often
   have particularly nice dynamical properties.   Bishop \cite{bishopclassS} has recently presented methods that allow the 
   construction of a vast array of 
   entire functions with bounded criticality.
   
When our functions do have asymptotic values, 
  or critical points of unbounded multiplicity,
  we will impose some geometric conditions concerning the singular value at $\infty$.
  Essentially, the following condition says 
 that the set of points where $f$ is large is 
 sufficiently
 thick near the real axis. 

\begin{defn}[Sector condition]
\label{def:sector}
Let $f$ be a real transcendental entire function and define 
  \[ \Sigma := \{\sigma\in\{+,-\}: 
     \text{there is some $x\in\R$ whose orbit accumulates on $\sigma\infty$}\}
           \]
We say that $f$ satisfies the \emph{sector condition} if, for every
 $M>0$ and $\sigma\in\Sigma$, there exist $\theta>0$ and $x_0>0$ such that
    \[ |f(\sigma x+iy)| > M \]
 whenever $x\geq x_0$ and $|y|\leq \theta x$. 
\end{defn} 

For $f\in\SR$, the sector condition is equivalent to requiring that there are
   constants $r,K>0$ such that
\begin{equation}\label{eqn:log_der}
 \frac{\vert f'(\sigma x)\vert}{\vert f(\sigma x)\vert}
\leq K\cdot\frac{\log\vert f(\sigma x)\vert}{x}
\end{equation}
for all $x\geq r$ and all $\sigma\in\Sigma$ \cite[Theorem 6.1]{wanderingdomains}. 
 It is 
satisfied for most explicit transcendental
entire functions that have finite order of growth, such as 
$z\mapsto ze^z$. In
\cite{wanderingdomains}, this condition is used to exclude the existence
of wandering domains for certain real transcendental functions. 

\begin{thm}[Density of real-hyperbolicity] \label{thm:densityunbounded}
 Let $n\geq 1$ and let $N$ be an $n$-dimensional (topological)
  manifold. Suppose that
  $(f_{\lambda})_{\lambda\in N}$ is a continuous family of functions
  $f_{\lambda}\in \SR$ such that the following three conditions hold:
 \begin{enumerate} 
  \item $\# S(f_{\lambda})\leq n$ for all $\lambda\in N$;
  \item no two maps $f_{\lambda_1}$ and $f_{\lambda_2}$ are conjugate
   by a map $M\in \MobR$;
  \item $f_{\lambda}$ has bounded criticality for every $\lambda\in N$, \emph{or}
      $f_{\lambda}$ satisfies the sector condition for every $\lambda\in N$.
 \end{enumerate}
 Then the set $\{\lambda\in N: f_{\lambda}\text{ is real-hyperbolic}\}$ is 
  open and dense in $N$. 
\end{thm}
\begin{remark}[Remark 1]
If $f$ is bounded along the real axis, then it trivially   
 satisfies the sector condition, so Theorem \ref{thm:densityunbounded}
 contains Theorem \ref{thm:densitybounded} as a special case.  
\end{remark}
\begin{remark}[Remark 2]
 Again, there is an analogous statement to Theorem \ref{thm:perturbbounded}:
  any map $f\in\SR$ that has bounded criticality or satisfies the sector condition can be 
  perturbed to a real-hyperbolic function $g\in\SR$ by pre-\ and post-composition
  with some $\phi,\psi\in\HomeoR$ close to the identity. 
\end{remark}

To describe some families to which the preceding result applies, let
$f\in\SR$, choose $\eps>0$ smaller than
one-half the minimal distance between two different singular values
of $f$, and set
  \[ W := \{z\in\C: \dist(z,S(f))<\eps\}. \]
Every component of $f^{-1}(W)$ is mapped either as a finite-degree
 branched covering or as an infinite-degree covering map by $f$. 
 We say that $f$ has \emph{$k$ singularities} if there are exactly
 $k$ components of $f^{-1}(W)$ on which $f$ is not one-to-one.
 (In particular, $f$ has at most $k$ critical points.)

If $f\in\SR$ has only a finite number of singularities, then $f$
 is of the form
  \[ f(z) = \int P(w)e^{Q(w)}dw, \]
where $P$ and $Q$ are real polynomials with $P\not\equiv 0$ and
$\deg Q \geq 1$. It is well-known that
such functions satisfy the sector condition; see
Lemma~\ref{lem:finitesingularities}.

\begin{cor}[Density of real-hyperbolicity]
\label{cor:thm1point7}
\begin{enumerate}
\item For each $k$, real-hyperbolicity is dense in the space of functions
 $f\in\SR$ that have $k$ singularities. 
\item
Real-hyperbolicity is dense in the family
   \[ S_{a,b}: x\mapsto ax e^x + b, \quad
              a\in\R\setminus\{0\}, b\in\R. \]
\end{enumerate}
\end{cor}

\subsection*{QC-rigidity for maps in $\SR$}

As is usual, our proof of the above results proceeds along three steps:
\begin{enumerate}
 \item \emph{QC rigidity:} Two functions that are topologically 
   (or combinatorially) conjugate
   are in fact quasiconformally conjugate;
 \item \emph{Absence of line fields}: The functions under consideration
   support no nontrivial quasiconformal deformations on the Julia set;
   \label{item:linefields}
 \item \emph{Parameter space arguments}: Density of hyperbolicity
  is deduced from the first two statements by performing suitable
  perturbations in parameter space. 
\end{enumerate}

Traditionally, the first step of this program has been the hardest to
achieve. In our context, we are able to solve it completely, i.e. {\em without}
assuming the sector condition or bounded criticality. This is accomplished by
combining the solution of the rigidity problem by the second author in joint work 
with Trevor Clark, see \cite{CS},
with recent results by the first author on  the dynamics of entire functions
near infinity \cite{boettcher}. 

\begin{thm}[QC Rigidity for maps in $\SR$]
\label{thm:qcrigidity}
Suppose that $f,g\in\SR$ are topologically conjugate on the complex plane,
and that the conjugacy takes the real axis to itself. 
Then $f$ and $g$ are quasiconformally conjugate.
\end{thm}

An immediate corollary is:

\begin{cor}[Connected conjugacy classes]\label{cor:conjconnected}
Take $f\in \SR$. Then the conjugacy class of $f$ (i.e. the set of maps 
that are topologically conjugate to $f$ on the complex plane) is connected with respect to the
topology of locally uniform convergence. 
\end{cor}

\begin{remark}
This statement is true even if one takes the topology coming from the
natural parameter space $\MfR$. (For the definition of this 
space see Section~\ref{sec:parameter}.)
\end{remark}

Step (\ref{item:linefields}) contains an additional
complication in the case of transcendental maps: it is necessary
to rule out the existence of invariant line fields on the set of 
escaping points as well as the set of points that tend to
escaping singular orbits under iteration. (Both sets are contained in the
Julia set.) While the first issue was resolved in \cite{boettcher}, 
we can deal with the second only by assuming either the
sector condition or bounded criticality.

\subsection*{Statement of corresponding results for circle maps and trigonometric polynomials}\label{subsec:trigo}

As usual in one-dimensional real dynamics, our results for real functions
have analogues for circle maps. Here it is natural to 
consider transcendental (non-rational) analytic self-maps of the
punctured plane $\CS := \C\setminus\{0\}$ that preserve the
unit circle. For such a function $f$, we can define the set of
singular values $S(f)\subset\CS$ analogously to the case
of entire functions. The natural class to consider for our purposes is 
\[ \SC := \{f:\CS\to\CS \text{ transcendental:} 
        \ f(S^1) \subset S^1, S(f)\subset S^1, \# S(f) < \infty \}. \] 
 We note that every map $f\in\SC$ has at least one critical
  point on the circle; see Lemma~\ref{lem:circlemapcrit}.
Again, $f\in \SC$ is called hyperbolic if every singular value belongs to a basin of attraction
of a periodic point in $S^1$.

\begin{thm}[Density of hyperbolicity for circle maps] \label{thm:densitycircle}
 Let $n\geq 1$ and let $N$ be an $n$-dimensional (topological) manifold. Suppose that
  $(f_{\lambda})_{\lambda\in N}$ is a continuous family of functions 
  $f_{\lambda}\in\SC$ such that
\begin{enumerate}
 \item $\# S(f_{\lambda})\leq n$ for all $\lambda\in N$ (recall that
  $S(f_{\lambda})\subset \C^*$ by definition; i.e.\ this count does
  not include $0$ or $\infty$);
 \item no two maps $f_{\lambda_1}$ and $f_{\lambda_2}$ are conjugate by a rotation.
\end{enumerate}
 Then the set $\{\lambda\in N: f_{\lambda}\text{ is hyperbolic}\}$ is open and dense in
  $N$.
\end{thm}

As before, there is an associated rigidity statement: 

\begin{thm}[QC rigidity for maps in $\SC$]
\label{thm:qcrigcircle}
Suppose that $f,g\in\SC$ are topologically conjugate on $\C^*$, and that the 
conjugacy 
preserves the unit circle. Then $f$ and $g$ are quasiconformally conjugate.
Furthermore, the dilatation of the map is supported on the Fatou set.
\end{thm}

A natural family of  degree $D$ circle maps that are $2m$-multimodal
can be described as follows. For $\mu \in  \mathbb{R}^{2m}$, 
consider the generalized trigonometric polynomial
\begin{equation}\label{trigonrm}
F_{\mu}(t)=D\cdot t+ \mu_1 + \mu_{2m}\sin(2 \pi mt)+ \sum_{j=1}^{m-1}( \mu_{2j}\sin(2\pi
jt) + \mu_{2j+1}\cos(2\pi jt)).
\end{equation} 
$F_\mu$ induces a circle map
$f_{\mu}:\mathbb{S}^1 \to \mathbb{S}^1$  (via the covering map 
$\P(t)=e^{2\pi i t}$). 
Note that if $\mu,\mu'\in \mathbb{R}^{2m}$
with $\mu_1-\mu_1'\in \Z$ and $\mu_j=\mu_j'$ for $j\neq 1$, then $f_\mu=f_{\mu'}$.
So it is natural to consider $f_\mu$ as parametrized by
$\mu =(\mu_1,\ldots,\mu_{2m}) \in \Delta$, 
where 
\begin{equation}\label{eqn:Delta}\Delta := 
\{\mu \in  \R/\Z\times \mathbb{R}^{2m-1} \ : \  \mu_{2m} > 0 \ \hbox{and}\  f_{\mu} \   \hbox{is } 
\   2m-\hbox{multimodal} \ \}.\end{equation}
More generally 
we could require that $f_\mu$ has precisely $2m$ critical points on the circle (counting multiplicities). Under these assumptions, 
  the map $f_\mu$ belongs 
  to the class $\SC$; see Lemma~\ref{lem:classical}. 
\begin{cor}[Density of hyperbolicity and rigidity in the trigonometric family]%
\label{cor:trigopol}%
 The set of parameters in $\Delta$ for which  $f_\mu$ is hyperbolic is dense. Furthermore, 
   let $\mu_0\in\Delta$.
 \begin{enumerate}
\item  Consider the set $[\mu_0]$ of parameters $\mu$ for which $f_{\mu}$ is topologically
 conjugate to $f_{\mu_0}$ by an order-preserving homeomorphism of the circle. Then 
 $[\mu_0]$ has at most $m$ components. 

\item If $f_{\mu_0}$ has no periodic attractors on the circle, then each component of 
$[\mu_0]$ is equal to a point. 
\end{enumerate}
\end{cor}

This answers Conjectures 1, 2 and 3 
posed by de Melo, Salom\~ao and Vargas in \cite{melo-etal};
 in particular, it establishes density of hyperbolicity in the Arnol'd family mentioned at the beginning of this introduction
 (for $D=1$ and $m=1$). 
In \cite{MR2576261} the family $F_{a,b}(x)=2x +a+b\sin(2\pi x)$, $a\in \R$, $b=1/\pi$,
was discussed.  In this case, the corresponding circle map
  $f_{a,1/\pi}$ has a single cubic critical point and belongs to
  $\SC$; see 
  Lemma~\ref{lem:trigoaffinerigid}. Thus Theorem~\ref{thm:densitycircle}
  implies that the set of values for which $f_{a,1/\pi}$ is hyperbolic
  is dense; this fact also follows already 
  from \cite[Theorem C]{LevStr:invent}. 
  When $b<1/\pi$, the critical points do not
  belong to the circle and $f_{a,b}\notin \SC$ is a covering
  map of degree $2$. In this case, by Ma{\~n}\'e's theorem
  there is a dense set of parameters for which $f_{a,b}$ 
  is hyperbolic as a map of
  the circle (i.e., expanding on the complement of the---potentially
  empty---union of attracting basins on the circle). 
  For
  $b>1/\pi$, the map $f_{a,b}$ has two critical points on the circle,
  and our results imply density of hyperbolicity as
  well as 
  various conjectures
  stated in \cite{MR844701} and \cite{MR1355627}, as we will 
  discuss in \cite{RS3}.

\medskip

We remark that the proofs can also be applied to obtain the corresponding
results for families of finite Blaschke 
products 
$$B(z)=e^{2\pi i a_0} z^{k_0} \prod_{j=1}^n \left( \dfrac{z-a_j}{1-\bar a_j z} \right)^{k_j}, \quad |a_1|,\dots,|a_n|<1,a_0\in \R,k_0\ne 0, k_j\in \Z$$
for which all critical values, apart from
 $0$ and $\infty$ (which have period $\le 2$), lie on the circle. Similarly, although we stated the results for entire functions only in the
  transcendental case for simplicity, they can also be applied verbatim to polynomial families~-- e.g.~for 
  two-dimensional families of quartic polynomials having two distinct critical values, both of which are real. 
 Of course, here there is no need to use tools
 from transcendental dynamics~-- the only new ingredients compared to  \cite{KSS1} are our discussion of parameter spaces and
 the rigidity theorem from
 \cite{CS}.

\subsection*{Further directions}
The rigidity results in this paper can also be used, similarly as in  \cite{BvS_mono}, to prove monotonicity of entropy in families of real
transcendental functions. For example, it can be deduced that 
 the topological entropy of maps within the family
$$\R\ni x\mapsto a\sin(2\pi x)\in \R$$
increases with $a\ge 0$. Similar results hold for families of trigonometric polynomials; these 
questions will be discussed in a sequel to this paper, \cite{RS3}.

Similarly, Theorem~\ref{thm:qcrigcircle} implies Conjecture B in \cite{MR1355627} for the family $f_{a,b}(x)=x +a+b\sin(2\pi x)$, $a\in [0,1)$. This conjecture states that the set of parameters $(a,b)\in (0,1)\times \R$ so that 
the rotation interval of $f_{a,b}$ is equal to a given interval with irrational boundary points, 
is equal to a single point. (It was already shown in \cite{MR1355627} that this set is contractible.)
We will discuss how this follows  in \cite{RS3}.  A similar kind 
of question was raised in \cite[Section 5]{MR2402913} for the family of double standard maps: $x\mapsto 2x+a+b\sin(2\pi x)$, $a\in (0,1)$ and will also be discussed in \cite{RS3}.  

\subsection*{Acknowledgment} We thank the referees
  for their careful reading of our paper, and a large number of helpful comments.  We are also
 grateful to all of those with whom we had interesting conversations related to this work, particularly Adam Epstein and Alexandre Eremenko for
 a discussion connected to the material in Appendix \ref{app:parameter}. 

\section{Preparatory definitions and remarks}
\label{sec:prep}

\subsection*{Organisation of the paper}
In the remainder of this section we will collect notation and
 some simple facts. 
 In Sections~\ref{sec:qsrigidity}--\ref{sec:rigidity}
 we  prove Theorem \ref{thm:qcrigidity},
 that  topologically conjugate entire functions in $\SR$ 
 are quasiconformally conjugate.
This relies on two deep results.
The first ingredient (Theorem~\ref{thm:realgidityvS}) is a 
theorem on real 
 analytic interval maps $f\colon [0,1]\to [0,1]$.
Assume that two such maps are topologically conjugate and that the conjugacy 
maps hyperbolic periodic points to hyperbolic points, and critical points to 
 critical points
 of the same order.  Then these maps are quasisymmetrically conjugate.
This follows from results of
Trevor Clark and the second author of the current paper \cite{CS} that apply 
in fact to much more general functions (even $C^3$ mappings). This work builds on earlier results of Kozlovski, Shen and van Strien 
 \cite{KSS1}.
The second ingredient (Theorem~\ref{thm:boettcher})
 uses rigidity of escaping dynamics for transcendental entire functions,
 a result which was proved by the first author in \cite{boettcher}. 
 In order to prove Theorem \ref{thm:qcrigidity}, we will show how to
 apply and combine these two ingredients in our setting.

We then show in Section~\ref{sec:absenceline} that two maps that
 are quasiconformally conjugate, via a conjugacy that is conformal on the
 Fatou set, are in fact
 affinely conjugate. 
 To do this, we show that the maps we consider cannot 
 carry measurable invariant line fields on their Julia sets.

In Section~\ref{sec:parameter}, we introduce a natural parameter 
 space $\MfR$, 
 and discuss kneading sequences and analytic invariants. Using our 
 rigidity results, these can be used to
 characterize conformal conjugacy classes within the family. In 
 Section~\ref{sec:densityhyperbolicity} 
we then derive density of hyperbolicity for the families  in $\SR$. 
In Section~\ref{sec:circlemaps} we  discuss how to adapt our results to circle maps.
In an appendix, we further clarify the structure of the parameter space $\MfR$.  

\subsection*{Definitions}
Throughout this article, with the exception of Section \ref{sec:circlemaps}, 
$f:\C\to\C$ will
be a transcendental entire function that maps the real line to itself. 
 We recall that $S(f)$ denotes the set of \emph{singular values} of $f$.
 
Let $\Crit_{\R}(f)$ denote the set of real critical points of $f$, and 
$\CV_\R(f):= f(\Crit_\R(f))$.
We say that $\alpha$ is a {\em real-asymptotic value} if  
$f(x)\to \alpha$ 
as $x\to \infty$ or as $x\to -\infty$. 
Let  $S_\R(f)$ be the {\em set of real-singular values} of $f|\R$, 
i.e. the union  of $\CV_\R(f)$ and the real-asymptotic values. 
For any $X\subset\C$, we define the orbit
  \[ O_f^+(X) := \bigcup_{n\geq 0}f^n(X). \]
 The \emph{postsingular set} of $f$ is defined as
   \[ \P(f) := \cl{O_f^+(S(f))}. \]
 We also denote the escaping set of $f$ by
  $I(f)=\{z:|f^n(z)|\to \infty\mbox{ as }n\to \infty\}$
  and set $\IR(f)=I(f)\cap \R$.

 Recall that $\SR$ denotes the class of real transcendental entire 
  functions for which $S(f)$ is a finite subset of the real axis.
  Also recall that $\HomeoR$ denotes the set of all homeomorphisms
  $\psi:\C\to\C$ that commute with complex conjugation and are increasing on the real axis, 
  and that $\MobR\subset\HomeoR$ consists of the affine maps
  $z\mapsto az+b$, $a>0$, $b\in\R$. If $\psi\in\HomeoR$ is quasiconformal, we call $\psi$ a
  \emph{real-quasiconformal homeomorphism}. 

 We denote Euclidean distance by $\dist$ and spherical distance by
  $\dist^{\#}$. If $z_0\in\C$ and $\eps>0$, then we denote by
    \[ B_{\eps}(z_0) := \{z\in\C: |z-z_0|<\eps \} \]
  the Euclidean ball of radius $\eps$ around $z_0$. We also
  denote the unit disk by $\D := B_1(0)$.

\subsection*{Quasiconformal maps and invariant line fields}
 Throughout the article, we assume familiarity with the
  theory of quasiconformal mappings of the plane;
  compare e.g.\,\cite{MR2241787}. 

 We also use the notion of \emph{invariant line fields}. This is
  a standard concept in holomorphic dynamics, but notation
  sometimes varies, so we give a concise summary here. 
 A \emph{measurable line field} on a measurable set $A\subset\C$ is
 a measurable function $\ell$ from $A$ to the projective plane.
 (More precisely, $\ell$ takes each point $z\in A$ to a point in
  the projective tangent bundle at $z$; i.e.\ it represents
  a measurable choice of a real line in the tangent bundle.) 

A line field is \emph{invariant} under $f$ if, for almost every $z$,
the pushforward of the tangent line $\ell(z)$ is given by
$\ell(f(z))$. In other words, for almost every $z$, 
 \[ \ell(f(z)) = f'(z)\cdot\ell(z)  \]
 (note that the derivative acts on tangent lines by multiplication as long as $z$ is not a critical point). 

Invariant line fields are related to invariant Beltrami differentials
 (ellipse fields): if $\mu$ is an invariant Beltrami differential with
 $\mu(z)\neq 0$ almost everywhere on $A$, then e.g.\ the direction of the
 major axes of the ellipses described by $\mu$ will provide an invariant
 line field. 
 Similarly, if $\ell$ is an invariant
 line field, we can find a corresponding non-zero  
 invariant Beltrami differential on $A$. (See also 
 \cite[Section 3.5]{mcmullenrenormalization}.)

In particular, we have the following fact: If $f$ and $g$
 are quasiconformally but not conformally conjugate, then
 there is an $f$-invariant 
 line field supported on some set of positive measure. 

\subsection*{The Koebe Distortion Theorem}
We 
frequently use
  the following classical theorem 
 in our proofs. 
\begin{thm}[Koebe Distortion Theorem]
\label{thm:koebe}
For any univalent map $f\colon \D\to \C$ and any $z\in \D$,
 \begin{align*}
 &|f'(0)|\dfrac{|z|}{(1+|z|)^2}\le |f(z)-f(0)| \le 
 |f'(0)|\dfrac{|z|}{(1-|z|)^2} \quad\text{and} \\
    &|f'(0)|\dfrac{1-|z|}{(1+|z|)^3}\le |f'(z)| \le 
 |f'(0)|\dfrac{1+|z|}{(1-|z|)^3}.\end{align*}
 In particular, 
 $f(\D)\supset B_{|f'(0)|/4}(f(0))$. 
\end{thm}
 For a proof, see for example \cite[Theorem 1.3]{pommerenke}.
 
  \subsection*{Functions with finitely many singularities and the
  sector condition} We note two standard facts regarding entire
  functions with finitely many singularities (compare
  \cite{elfving}),
 which show that Corollary \ref{cor:thm1point7} indeed follows from Theorem \ref{thm:densityunbounded}. 
 
\begin{lem}[Functions with finitely many singularities] \label{lem:finitesingularitiesform}
 Suppose that $f$ is a real transcendental entire function. Then $f$
  has only finitely
  many singularities if and only if there are real polynomials $P$ and $Q$
  with $P\not\equiv 0$ and $\deg Q \geq 1$ such that
    \[ f'(z) = P(z)e^{Q(z)}. \] 
\end{lem}
\begin{sketch}
 First suppose that $f$ has only finitely many singularities. Then
  $f'$ has only finitely many zeroes. So if we let $P$ be  
  a real polynomial having the same zeroes as $f'$ (counting
  multiplicities), we can write 
   \[ f'(z) = P(z)e^{g(z)}  \] 
  for some nonconstant 
  real entire function $g$. 
The function $g$ cannot be transcendental, as otherwise one could show
  that the function $f$ has infinitely many
  singularities. So $g$ must be a real polynomial. 

 The converse is trivial, as one can check by hand that any function
  of the stated form has only finitely many singularities; compare
  (\ref{eqn:asymptoticsfinitesingularities}).
\end{sketch}

\begin{lem}[Sector condition]\label{lem:finitesingularities}
Let $f\in\SR$ be a function of the form
  \[ f(z) = \int P(w)e^{Q(w)}dw, \]
where $P$ and $Q$ are real polynomials with $P\not\equiv 0$ and
$\deg Q \geq 1$. 

Then $f$ satisfies the sector condition (Definition \ref{def:sector}).
\end{lem}
\begin{sketch}
This can be checked by direct calculation. Indeed, the
 function $f$ satisfies
   \begin{equation}
   \label{eqn:asymptoticsfinitesingularities}
   f(z) = \left(\frac{P(z)}{Q'(z)}+O\bigl(|z|^{\deg(P)-\deg(Q)}\bigr)\right)
               e^{Q(z)} + O(1) \end{equation}
 as $z\to\infty$. (See  \cite[Lemma 4.1]{martinthesis}) 
 The claim follows easily from this estimate.  
\end{sketch}

 \subsection*{Explicit families}
 To conclude this section, we collect some simple facts that are needed to 
  deduce Corollaries \ref{cor:trigopolreal} and \ref{cor:trigopol}, concerning
  explicit families of trigonometric polynomials and circle maps, from the
  more general Theorems \ref{thm:densitybounded}, \ref{thm:densitycircle} and \ref{thm:qcrigcircle}.
  These results are all well-known and easy to prove, but we include
  the short arguments for completeness.

 \begin{lem}[Cosine maps and standard maps] \label{lem:cos}
  Let $(a,b),(c,d)\in\R^2\setminus\{(0,0)\}$ with
    $(a,b)\neq (c,d)$. Then the cosine maps
   $C_{a,b}$ and $C_{c,d}$ are not conjugate by an affine map $M\in\MobR$. 

  The analogous statement holds for the family $S_{a,b}$. 
 \end{lem}
 \begin{proof}
  We prove the contrapositive, so suppose that 
   $C_{c,d}= M\circ C_{a,b} \circ M^{-1}$ for some affine map
   $M(x)=\alpha x + \beta$, $\alpha>0$. Since both maps have period
   $2\pi$, we must have $\alpha=1$. Furthermore, we note
 that 
    the image of the real axis under a map $C_{a,b}$ is an interval that
    is symmetric around the origin. This implies that $\beta=0$, and hence
    $M=\id$ and $(a,b)=(c,d)$. 

  We note that $S_{a,b}=axe^x+b$ has a critical point at $-1$ and an asymptotic value
   at $b$, and no other critical or asymptotic values. Furthermore, $z=0$ is the
   unique preimage of the asymptotic value $b$. If we have a conjugacy $M\in\MobR$ between
   $S_{a,b}$ and $S_{c,d}$, it follows that $M$ fixes $0$, $-1$ and $\infty$. Hence
   $M=\id$ and $(a,b)=(c,d)$. 
 \end{proof}

\begin{lem}[Number of critical points of Arnol'd-type maps]
\label{lem:classical}Let $F_\mu$ be a (generalized) 
trigonometric polynomial as in (\ref{trigonrm}).
Then the corresponding circle map 
$f_\mu$ has exactly $2m$ critical points in $\CS$, counted with multiplicities. 
Moreover, $f_\mu$ has no asymptotic values in $\CS$. 
\end{lem}
\begin{proof}
This is a classical fact. Indeed,
note that $F'_{\mu}$ is a trigonometric polynomial of degree $m$, and hence
 \[ F'_{\mu}(z) = R(e^{2\pi i z}), \]
  where $R$ is a rational function of degree $2m$. Thus $F_{\mu}$ has exactly $2m$ 
  critical points in every vertical strip of width $1$ (counting multiplicities). The claim follows.
  
  It is also elementary to see that $F_{\mu}$ has no finite
    asymptotic values (and hence $f_{\mu}$ has
   no asymptotic values in $\CS$). Let us sketch the argument.
 Suppose by contradiction that $\gamma:[0,\infty)\to\C$ was a curve to
   infinity with $F_{\mu}(\gamma(t))\to a\in\C$. We must have $|\im\gamma(t)|\to\infty$.
   If $D=0$,  this follows from periodicity and otherwise from the fact that
     $F_{\mu}(z)=Dz + O(1)$ when restricted to any horizontal strip.
     Similarly, we must have $|\re\gamma(t)|\to\infty$, as 
   $|F_{\mu}(z)|$ grows like $|\mu_{2m}|\cdot e^{2\pi m|\im z|}/2$ in any vertical strip. 
 
  For $\zeta\in\gamma$, we can write
      \[ F_{\mu}(\zeta) = D\zeta + \mu_{2m}e^{\pm 2\pi i m\zeta}/2i + o(e^{2\pi i m\zeta}). \]
    For sufficiently large $\zeta$, the argument of $\zeta$ will be contained in a fixed interval
      of length $\pi/2$, while the second term keeps ``spiralling'' to 
  infinity. It follows that we must have
       $\limsup |F_{\mu}(\gamma(t))|=\infty$, a contradiction.
       
    (The claim can also be deduced directly from the celebrated Denjoy-Carleman-Ahlfors theorem.)
\end{proof}

  \begin{lem}[Conformal conjugacy classes]
\label{lem:trigoaffinerigid}
   Let $D\ge 0$, $m\ge 1$ be integers. 
  Consider  trigonometric polynomials as in  (\ref{trigonrm})
and let $f_\mu$ be the corresponding map of the circle $S^1$. 
Suppose that $M(z)=e^{2\pi i \beta}z$ 
  is a rotation. Then $M$ conjugates the map  
  $f_\mu$ to some map $f_{\mu'}$ in the same family if and only if 
 $\beta=p/m$ where $p\in \Z$.
 In particular, each affine conjugacy class (via rotations)
 consists of at most $m$ maps.
 \end{lem}
\begin{proof}
Assume that $f_\mu,f_{\mu'}$ are conjugate by a rotation $M(z)=e^{2\pi i\beta}z$.
Note that the lift of 
$M\circ f_\mu\circ M^{-1}(t)$ is equal to 
\begin{align*}
 F_{\mu}(t-\beta) + \beta =
  Dt - (D-1)\beta &+ \mu_1 + \mu_{2m}\sin(2 \pi m(t-\beta)) 
\\ &+ 
\sum_{j=1}^{m-1}\left( \mu_{2j}\sin(2\pi j(t- \beta)) + \mu_{2j+1}\cos(2\pi j(t-\beta))\right).
\end{align*}
Using the addition theorems for sine and cosine, we see that
 this map is in the form  (\ref{trigonrm}) if and only if $m\beta=0\ \mod 1$.  
The lemma follows.
\end{proof}

\begin{remark}[Remark 1]
Consider $t\mapsto 3t+\sin(4\pi t)+\epsilon \sin(2\pi t)$.
Conjugating this with $t\mapsto t+1/2$
gives the map $t\mapsto 3t+1+\sin(4\pi t)- \epsilon \sin(2\pi t)$. 
These maps are both close
 to $t\mapsto 3t+\sin(4\pi t)$  as circle maps,
so taking the quotient of the set $\Delta$ defined in~\eqref{eqn:Delta} by conjugacy classes results in a space with
an orbifold structure, not a manifold structure. 
\end{remark}

\begin{remark}[Remark 2]
If $D\ne 1$, then for each map $F_\mu$ as in (\ref{trigonrm}),
one can find $M\in \MobR$ so that $M\circ F_\mu \circ M^{-1}$ is equal to 
\begin{equation}
t\mapsto Dt+ \sum_{j=1}^{m}( \mu'_{2j}\sin(2\pi
jt) + \mu'_{2j-1}\cos(2\pi jt))\label{eq:trigo-alter}
\end{equation}
by taking $M(t)=t+\mu_1/(D-1)$. (And, vice versa,  each map 
as in (\ref{eq:trigo-alter})  can be affinely conjugated to one with as in  (\ref{trigonrm}) 
by a translation $M(t)=t+\beta$ with $\beta$ chosen 
so that $\mu_{2m-1}\cos(2\pi \beta)+\mu_{2m}\sin(2\pi \beta)=0$.) 
\end{remark}

 \begin{lem}[Trigonometric polynomials] \label{lem:trig}
  Suppose $f$ is a trigonometric polynomial of degree $n$
   as in (\ref{eqn:trig}), and let
   $\phi_n,\psi_n\in\HomeoR$ 
  with
  $\phi_n,\psi_n\to \id$ such that
   $g_n := \psi_n\circ f \circ \phi_n^{-1}$ are entire functions for all $n$. 

  Then there is a sequence $\alpha_n>0$ with $\alpha_n\to 1$ such that
   $f_n=g_n(\alpha_n z)/\alpha_n$ is a trigonometric polynomial for every $n$.
   (Furthermore, since
    $f_n\to f$, the Fourier coefficients of $f_n$ converge to those of $f$.)
 \end{lem}
\begin{proof}
Let $n\in\N$ and define $\theta_n(z)=\phi_n(\phi_n^{-1}(z)+2\pi)$. Then
  $\theta_n$ is a homeomorphism. 
 For purposes of legibility we suppress the subscript $n$ in the following.
 Note that 
  \[ g(\theta(z))=g(\phi(\phi^{-1}(z)+2\pi))=
     \psi(f(\phi^{-1}(z)+2\pi))=\psi(f(\phi^{-1}(z)))=g(z). \]
   It follows that $\theta$ is holomorphic, and hence an affine map
   $\theta(z)=z+\beta$, where $\beta=\beta_n=\theta_n(0)\to 2\pi$. 

   So each $g_n$ is periodic with period $\beta_n$, and we are done if we set
   $\alpha_n := \beta_n/2\pi$. 
\end{proof}

\section{Quasisymmetric rigidity on the bounded part of the real dynamics}
\label{sec:qsrigidity}
In this section we consider the following class of functions, which is more general than those considered in the introduction. 

\begin{defn}[The class $\BR$]
 We denote by $\BR$ the set of all real transcendental entire functions
  with bounded singular sets. (Note that we do \emph{not} require 
  that all singular values are real.)
\end{defn}

If $f\in \BR$, then either $\lim_{x\to +\infty} |f(x)|\to \infty$ or $\sup_{x\ge 0}|f(x)|<\infty$
(and similarly either $\lim_{x\to -\infty} |f(x)|=\infty$  or $\sup_{x\le 0}|f(x)|<\infty$). 
For this class of functions one has the following:

\begin{lem}\label{lem:ahlfors}
Let $f\in\BR$, and let $\sigma\in\{+,-\}$. Suppose that
  $\lim_{x\to\sigma\infty} |f(x)|=\infty$. Then
  \[ \liminf_{x\to\sigma\infty} \frac{\log \log |f(x)|}{\log|x|} \geq \frac{1}{2}. \]
\end{lem}
\begin{proof}
This is a standard consequence of the Ahlfors distortion theorem as stated in
  \cite[Corollary to Theorem 4.8]{ahlforsconformal}. Indeed, let $M>0$ be chosen sufficiently large to ensure that 
   $S(f)\cup\{f(0)\}\subset B_M(0)$, and let $V\subset\C\setminus\{0\}$ be the component of the preimage of
   $E(M) := \C\setminus \overline{B_M(0)}$ that contains $\sigma x$ for sufficiently large $x$. 
   Then $f:V\to E(M)$ is a universal covering map, and by the Ahlfors distortion 
   theorem, we have 
    \[ \log | \log f(z)| \geq \frac{1}{2}\log|z| + O(1) \]
    for $z\in V$, where $\log f$ is a branch of the logarithm of $f|_V$. 
    The claim follows, since
    $|\log f(x)|=\log |f(x)| +O(1)$ as $x\to \sigma\infty$.
  (Compare \cite[Formula (1.2)]{MR2885574} for further discussion.)
\end{proof}


 Hence we see that, if $f(\sigma x)$ is unbounded as $x\to+\infty$, then
   $|f(\sigma x)|>2x$ for sufficiently large $x$. In particular, either 
   \begin{itemize}
     \item $|f^2(\sigma x)|$ is bounded
     as $x\to + \infty$ (which can occur either if $|f(\sigma x)|$ remains bounded, or if
     $f(\sigma x)\to -\sigma\infty$ and $f(-\sigma x)$ remains bounded); or
    \item $|f^n(\sigma x)|\geq 2^n|x|$, and hence $\sigma x\in I(f)$, for sufficiently large $x$.
  \end{itemize}

\subsection*{Combinatorial conjugacy}

\begin{defn}[The partition $\Part(f)$] \label{def:partition}
 Let $f\in\BR$.  We denote by $\Part(f)\subset\R$ the set consisting of
  \begin{enumerate}
   \item the real critical points  $\Crit_\R(f)$ of $f$,
   \item the real hyperbolic attracting periodic points of $f$ and 
   \item the real  parabolic  periodic points of $f$.
  \end{enumerate}
\end{defn}

\begin{defn}[Combinatorial conjugacy on the real line]
\label{def:Rcomb-equiv}
Two functions $f,g\in\BR$ are called \emph{combinatorially conjugate on the real line} if
 there is an order-preserving bijection
  \[ h: \Part(f)\cup O_f^+(S_{\R}(f)) \to             \Part(g)\cup O_g^+(S_{\R}(g))\]
 that satisfies $h \circ f = g\circ h$, maps points as in (a)-(c) above to corresponding
 points and    preserves the degree of critical points. Furthermore, asymptotic
 values in $S_{\R}(f)$ and $S_{\R}(g)$ should correspond to each other, in the following
 sense: for $\sigma\in\{+,-\}$, we have $\lim_{x\to\sigma\infty} f(x)=a\in \R$ if and only if 
  $\lim_{x\to\sigma\infty}g(x) = h(a)\in \R$. 
\end{defn}

\begin{lem}[Combinatorial conjugacy and topological conjugacy on the real line]
\label{lem:comb-conj}
If $f,g\in \BR$ are combinatorially conjugate on the real line then they are topologically conjugate on the real line.
Moreover, the topological conjugacy $h$ extends the combinatorial conjugacy and hencesatisfies the following properties:
\begin{enumerate}
\item for each $n\ge 1$ and each $x\in \R$, $x$ is a critical point of $f$ of order $n$ 
iff $h(x)$ is a critical point of $g$ of order $n$
and 
\item $x\in \R$ is a parabolic periodic point of $f$ iff $h(x)$ is a parabolic periodic point of $g$.
\end{enumerate}
 Furthermore, the extension $h$ is uniquely determined outside of
  the union of real attracting and parabolic basins. 
\end{lem}
\begin{proof} 
Since $f\in \BR$,  Lemma~\ref{lem:ahlfors} implies that if 
$\lim_{x\to \infty}f(x)=\infty$ then 
$f\colon \R\to \R$ can be extended 
to a continuous map $\hat f\colon (-\infty,\infty]\to (-\infty,\infty]$ having $\infty$ as an attracting fixed point.
More generally, if $f^2|\R$ is unbounded, then either $\lim_{x\to \infty} f^2(x)=\infty$, 
$\lim_{x\to -\infty} f^2(x)=-\infty$ or both. So in the latter case,
we can extend $f$ to a continuous map
$\hat f\colon(-\infty,\infty]\to (-\infty,\infty]$, to $\hat f\colon [-\infty,\infty)\to [-\infty,\infty)$
or to $\hat f\colon [-\infty,\infty]\to [-\infty,\infty]$ having respectively, $\infty$, $-\infty$
or $-\infty,\infty$ as attracting fixed points or attracting periodic two points. 
It follows that the only difference between a map $f\in \BR$ and a multimodal map
on a compact interval is that $f$ can have infinitely many turning points. 
Recall that a point $c$ is called a {\em turning point} of an interval map $f\colon I\to I$
if the map has a local extremum at $c$ and $c$ is in the interior of $I$.
The assumption in Definition \ref{def:Rcomb-equiv}
 about the way asymptotic values are mapped by
 $h$ ensures, furthermore, that two combinatorially conjugate maps extend in the
 same manner. 

So assume that $f,g\in \BR$ are combinatorially conjugate on the real line. 
Since  $f$ and $g$ are real analytic,  
\begin{enumerate}
 \item[(i)] $f$ and $g$ have no wandering intervals (see \cite[Theorem IV.A]{demelovanstrien});
 \item[(ii)] $f$ and $g$ have at most finitely many real periodic points that
 are not hyperbolic and repelling  (see \cite[Theorem IV.B']{demelovanstrien});
in particular, $f$ and $g$ have  no 
     intervals consisting entirely of periodic points;
 \item[(iii)] each periodic turning point is attracting. 
\end{enumerate}
(Note that point (i) implies that any extension of $h$ as in the statement of the lemma
 is
 uniquely determined outside of attracting and parabolic basins.)

Let us denote the union of the
immediate basins of the finitely many real periodic attractors of $f$ and $g$ by $B_0(f)$
and $B_0(g)$, respectively. Since the combinatorial conjugacy $h$ sends periodic (parabolic) attractors
to periodic (parabolic) attractors, we can extend $h$ to a conjugacy 
between $f\colon B_0(f)\to B_0(f)$ and $g\colon B_0(g)\to B_0(g)$
(mapping iterates of singular values to corresponding iterates of singular values).
This implies that assumption (iv) of  \cite[Theorem II.3.1]{demelovanstrien} 
is also satisfied, and one can easily check that the proof of that theorem
goes through verbatim in our context.
  (Alternatively, we could apply the latter theorem directly to
   a restriction of $f$ or a modification of such a restriction as in
   the proof of Theorem \ref{thm:realrigidity} below.) 
\end{proof}

\subsection*{Quasisymmetric rigidity}

One of the main technical ingredients in this paper is the following:
\begin{thm}[Quasisymmetric rigidity on the bounded part of the real dynamics] \label{thm:realrigidity}
Let $f\in\BR$.
Then there exists a compact interval $J\subset\R$ (possibly empty or consisting of only one point)
 with the following properties.
 \begin{enumerate}
 \item If $x\in J$ and $f(x)\notin J$, then $x\in \IR(f)$.
  \item For every $x\in\R$, either $x\in \IR(f)$ or 
    $f^j(x)\in J$ for all $j\geq 2$.
  \item The set of points $z\in\C$ whose $\omega$-limit set is contained in
   $J$ and which do not belong to an attracting or parabolic basin has empty interior
   and does not
   support any invariant line fields.
  \item If $g\in\BR$ is combinatorially conjugate on the real line to
     $f$, then there is an interval $\tilde J$, which has
     the corresponding properties for $g$,  and a
     quasisymmetric conjugacy between $f|_{J}$ and $g|_{\tilde J}$ that agrees with the
     combinatorial conjugacy. 
 \end{enumerate}
\end{thm}

This theorem is essentially proved in \cite{CS}, but the setting there is slightly
different from ours (in \cite{CS}, the functions have compact domains).
Hence the remainder of this section is devoted to showing
how to obtain the required intervals $J$ and $\tilde J$ under the assumptions of Theorem \ref{thm:realrigidity}, using the
 results of \cite{CS}.

\subsection*{Anchored interval maps}\label{subsec:anchored}

\begin{defn}[The class ARAIM of anchored maps]
Let $a,b\in\R$, $a<b$, and let $f:[a,b]\to\R$ be real-analytic (by which we mean that
 $f$ is real-analytic on an open interval containing $[a,b]$). If
 $f(\{a,b\})\subset\{a,b\}$, then (following \cite{milnortresser}) we say that $f\colon [a,b]\to \R$ is an
 \emph{anchored real-analytic interval map} (ARAIM).
\end{defn}

An ARAIM $f\colon [a,b]\to \R$ and an ARAIM $g\colon [\tilde a,\tilde b]\to\R$ 
are said to be {\em topologically conjugate} if there exists an order-preserving homeomorphism
$h\colon [a,b]\to [\tilde a,\tilde b]$ so that 
$h\circ f(x)=g\circ h(x)$ for each $x\in [a,b]$ for which $f(x)\in [a,b]$ or
$g(h(x))\in [\tilde a,\tilde b]$.
In \cite{CS}, the following rigidity result is established.

\begin{thm}[Quasisymmetric rigidity]\label{thm:realgidityvS}
Suppose that $f$ and $g$ are ARAIM, and that $f$ and $g$ are  topologically conjugate
via a conjugacy $h$. 
Assume moreover that \begin{enumerate}
\item for each $n\ge 1$ and each $x\in \R$, $x$ is a critical point of $f$ of order $n$ 
iff $h(x)$ is a critical point of $g$ of order $n$
and 
\item $x\in \R$ is a parabolic periodic point of $f$ iff $h(x)$ is a parabolic periodic point of $g$.
\end{enumerate}
 Then the topological conjugacy between $f$ and $g$ extends to a quasisymmetric homeomorphism
 on the real line.
 \end{thm}

This theorem was announced by the second author in 2009. Since  then,
   a significantly stronger result than Theorem~\ref{thm:realgidityvS} was established in joint work with Trevor Clark, 
   and partly with Sofia Trejo, 
  and hence the original manuscript remains unpublished. In particular, one of the key ingredients in the argument, 
  the proof of \emph{complex bounds} (even for $C^3$ maps),  has been 
  substantially unified and simplified and appears in \cite{CST}. The work on rigidity has also been 
  extended to the setting of $C^3$ mappings, and only this more
  general result \cite{CS} will be submitted for publication (in the near future).
   
 Since the latter manuscript is not yet available, let us comment briefly on the proof of Theorem
   \ref{thm:realgidityvS}. We emphasize that, under the additional assumption that $f$ and $g$ are polynomials without parabolic periodic points whose
   critical points all have even order and are all real, this theorem already appears in
  \cite{KSS1} (see the \emph{Rigidity Theorem} and \emph{Rigidity Theorem'} on page 751 of that paper).
  The proof there uses the following steps: 
\begin{enumerate}
  \item associate to both $f$ and $g$ an induced first return map, called a \emph{complex box mapping}, near iterates of the critical points;
  \item the main step in the proof is then to show that for a certain subsequence (the {\em enhanced nest})  these complex box mappings satisfy some a priori \emph{complex bounds}
   (these are uniform estimates on the moduli of certain annuli);
  \item use the \emph{qc-criterion} from \cite[Appendix]{KSS1} 
   to show that these a priori bounds imply
   that these first return maps are quasiconformally conjugate;
  \item  spread this quasiconformal conjugacy to the entire complex plane. 
\end{enumerate}
In \cite{CST} it is shown that such complex bounds hold
 for arbitrary real analytic maps (and even more generally), a fact which was 
 known previously in a large class of special cases, see \cite{KSS1} and \cite{Shen}.
  Exactly as in \cite{KSS1}, we can then use the qc-criterion 
which states that bounded geometry implies quasiconformal rigidity
  to    deduce that the complex box mappings  are quasiconformally conjugate. 
 (The proof of this criterion in \cite{KSS1} exploits recent results on quasiconformal maps due to Heinonen and Koskela \cite{heinonenkoskela}.
We note that, prior to \cite{KSS1}, their theorem and its variations were used to prove rigidity results in \cite{Przytycki-Rohde}, \cite{Haissinsky}, \cite{Graczyk-Smirnov}, \cite{LevStr:invent} and \cite{smania}, where in the last work, the author explicitly 
stated that a bounded shape property of puzzle pieces implies rigidity for non-renormalizable unicritical maps.)

Once this is done, we can extend the conjugacy to the real line in a quasisymmetric manner. Here the fact that the map is non-polynomial means that one no longer can use 
 B\"ottcher coordinates, and so one needs to paste together  (essentially by hand)
 the conjugacies which are constructed near different critical points.  (This final step is carried out in detail in \cite{CS}, where additional arguments are required to deal also with the $C^3$ setting.)

\medskip

We shall also require a result on the absence of invariant line fields. 
As mentioned above, the existence for complex box mappings for real-analytic maps is proved in \cite{CST}
(and in many cases in previous papers).  The fact that such a complex box mapping does not support invariant line fields is by now well-known,  see for example \cite{KSS1,MR2533666,CST} and also \cite{CS}.
Thus we immediately obtain the following theorem.

\begin{thm}[Absence of invariant line fields]\label{thm:realrigidity-linefield}
Let $f:[a,b]\to\R$ be an ARAIM, and let $U\supset[a,b]$ be an open subset of $\C$ on which
 $f$ is analytic. Then the set of points $z\in U$ for which $\dist(f^n(z),[a,b])\to 0$
 as $n\to\infty$ and that do not belong to attracting or parabolic basins has empty
  interior and does not support
 any invariant line fields. 
\end{thm}

A function $f\in\BR$ which is unbounded both for $x\to+\infty$ and $x\to-\infty$ has
a restriction that is an ARAIM, and hence in this situation Theorem \ref{thm:realrigidity} follows from
the statements above. 
In the case where $f$ is bounded to the left or to the right,
this is not necessarily the case, so we will need to be slightly more careful in showing how to
deduce Theorem \ref{thm:realrigidity}.
However, there are no new dynamical 
phenomena in this setting, and we will show that we can modify $f$ outside an interval
that contains all relevant dynamics to obtain an ARAIM. (Instead, we could also observe that
the proof from \cite{CS} goes through  in this slightly modified setting.)

\begin{proof}[Proof of Theorem \ref{thm:realrigidity} from Theorems~\ref{thm:realgidityvS} and \ref{thm:realrigidity-linefield}]
By Lemma~\ref{lem:comb-conj}, the maps $f$ and $g$ from assumption (c)
in the statement of the theorem are in fact topologically conjugate.
Let us distinguish a few cases.
\begin{claim}[Case 1]
$\R\setminus \IR(f)$ contains at most one point.
\end{claim}
In this case the set $J=\R\setminus\IR(f)$ satisfies all the
 requirements of the theorem. 
So from now on we assume in the proof that $\R\setminus\IR(f)$ 
 contains several points.

\begin{claim}[Case 2]
 $f$ is unbounded in both directions on the real line.
\end{claim}
In this case, let $a$ and $b$ be the smallest resp.\ largest
 non-escaping points under $f$. Then clearly $f(\{a,b\})\subset\{a,b\}$, so
 if we set $J := [a,b]$, then the restriction $f|_{J}$ is a an ARAIM. 
 So the theorem follows  from Theorems~\ref{thm:realgidityvS} and
 \ref{thm:realrigidity-linefield}.

\smallskip

It remains to deal with the cases where at least one of $f|_{(-\infty,0]}$ 
  and $f|_{[0,+\infty)}$ is bounded.

\begin{claim}[Case 3]
$f|_{\R}$ is bounded.
\end{claim}
In this case, there may not exist suitable points $a,b\in \R$
so that $f|_{[a,b]}$ 
becomes an ARAIM. Therefore we will modify
$f$ as follows.
Set
  \[ \alpha := \inf_{x\in\R} f(x)\quad\text{and}\quad \beta := \sup_{x\in\R} f(x) \]
and choose numbers $A<\alpha-1$ and $B>\beta+1$ that are not critical
points of $f$.

Choose $\eps>0$ such that
  \[ A+1<\re f(z) < B-1 \]
whenever
 \[ z\in U := \{x+iy: x\in [A,B], |y|<\eps\}. \]
We may also assume that $\eps>0$ is chosen sufficiently small that
$f$ is injective on the boundary segments $[A-i\eps,A+i\eps]$ and
$[B-i\eps,B+i\eps]$. 

Set $C := A-1$ and $D:= B+1$. We now define a quasiregular
extension $\tilde{f}:V\to \C$ of the restriction $f|_U$, where
  \[ V := \{x+iy: x\in [C-\eps,D+\eps], |y|<\eps\}. \]

This extension will be chosen to have the following properties:
\begin{enumerate}
 \item $\tilde{f}$ commutes with complex conjugation;
 \item $\tilde{f}(\{C,D\}) \subset  \{C,D\}$; \label{item:endpoints}
 \item $\tilde{f}$ is monotone (without critical points)
         on $[C-\eps,A]$ and $[B,D+\eps]$; \label{item:monotonicity}
 \item If $\tilde{f}$ is not holomorphic at $z\in V$, then
        $\re f(z)\in [A,B]$. 
\end{enumerate}

Such an extension is simple to construct. Indeed, we first determine
$\tilde{f}(C)$ and $\tilde{f}(D)$ according to 
(\ref{item:endpoints}) and (\ref{item:monotonicity}). Then we choose
$\tilde{f}$ to be a linear map on
$[C-\eps,C+\eps]\times [-\eps,\eps]$ whose image is
$[\tilde{f}(C)-1,\tilde{f}(C)+1]\times [-1,1]$, and similarly for
$D$. Finally, we use a diffeomorphism to interpolate between
this map and $f|_U$. 

Note that $f([A,B])\subset [A,B]$ so that the orbit of any $z\in V$ enters the region where
$\tilde{f}$ is not holomorphic at most once under iteration of
$\tilde{f}\colon V\to \C$. This means that $\tilde{f}$ has an invariant Beltrami
field on $V$. Extend this Beltrami field $\mu$ to $\C$ by setting it to zero outside $V$.
Now use 
 the Measurable Riemann Mapping Theorem
 to straighten $\tilde f$ to an analytic map $F$; i.e.\  
let $F=h^{-1}_\mu\circ \tilde f \circ h_\mu$ where $h_\mu$
is so that $\bar \partial  h / \partial h=\mu$. 
Then $F$ is holomorphic and an ARAIM, when restricted to a 
suitable interval $[a,b]$. Furthermore, the conjugacy 
 $h_{\mu}$ 
between
$F$ and $\tilde{f}$ (and hence $f$) is conformal on $U$, so it follows
from Theorem \ref{thm:realrigidity-linefield} that $f$ supports no invariant line fields
on the set of points whose $\omega$-limit set is contained in
$[\alpha,\beta]$. 

It is also clear that we can apply
 the same procedure to a function that
 is topologically conjugate on the real line to $f$ to obtain an 
 ARAIM that is
 topologically  conjugate on the real line to $F$. 
Hence we can apply Theorem \ref{thm:realgidityvS}.
This completes the proof of the theorem in the case where $f$ is bounded.

\begin{claim}[Case 4]
$f$ is unbounded in one direction, and bounded in the other.
\end{claim}
Let us assume without loss of generality that
$|f(x)|\to \infty$ as $x\to+\infty$ and that $\limsup_{x\to -\infty}|f(x)|<\infty$. If $f(x)\to-\infty$ as $x\to+\infty$,
then $f^2$ is bounded, and we can apply the previous argument to this iterate.
Hence we may suppose that $f(x)\to+\infty$ as 
$x\to +\infty$, in which case we see as above by 
Lemma~\ref{lem:ahlfors} that  $f(x)\in I(f)$ for sufficiently large $x$. Let $b\in\R$ be the largest
real nonescaping point of $f$; then $b$ is a (repelling or parabolic)
fixed point. We now distinguish three further subcases.
\begin{enumerate}
\item[(i)] If $\liminf_{x\to-\infty} f(x) > b$, we can choose $a$ as the
smallest preimage of $b$. Then $f|_{[a,b]}$ is an ARAIM, and $J:=[a,b]$ has the desired properties. 
\item[(ii)] If $f$ has
infinitely many preimages of $b$ on the real axis, we can pick $a$ as such
a preimage chosen small enough that $a<\alpha := \inf_{x\in\R} f(x)$. Again,
we can set $J:=[a,b]$ and $f|_{J}$ is an ARAIM.
\item[(iii)] In the remaining case, $f(x)<b$ whenever $x$ is sufficiently negative. We can choose
$A < \alpha - 1$ in such a way that $A$ is not a critical point and modify the 
function $f$ to the left of $A$, exactly as above, to obtain a quasiregular map that
straightens to a holomorphic map whose restriction to a suitable interval is an ARAIM.
So we are done also in this case, setting $J:=[A,b]$. \qedhere
\end{enumerate}
\end{proof}

\section{Gluing and extending quasiconformal homeomorphisms dynamically}
\label{sec:qcextension}

\subsection*{Topological equivalence} To ensure that not only the order relation
of the critical points and critical values of $f$ and $g$ on the real line are the same,
but that they are also compatible in the complex plane we 
use the notion of topological equivalence from \cite{alexmisha}.

\begin{defn}[Real-topological equivalence]\label{def:MfR}
 Two maps $f,g\in\SR$ are called \emph{real-topologically equivalent}
  if there are functions $\phi,\psi\in\HomeoR$ such that 
   \[ \psi(f(z)) = g(\phi(z)) \]
  for all $z\in\C$. 

 The set of all functions $g$ that are real-topologically equivalent to $f$ is denoted by
    $\MfR$. 
\end{defn} 
\begin{remark}[Remark 1]
 If $f$ and $g$ are real-topologically equivalent, then they are in fact
  \emph{real-quasicon\-for\-mally equivalent}; i.e.\ the maps $\psi$ and $\phi$
  can be chosen to be quasiconformal. Indeed, suppose that 
  maps $\phi,\psi\in\HomeoR$ as in the definition are given. Because $S(f)$ is finite,
  we can find a quasiconformal homeomorphism $\tilde{\psi}\in\HomeoR$ such that
  $\psi$ and $\tilde{\psi}$ are isotopic relative $S(f)\cup\infty$. We can lift the
  homotopy to a homotopy between $\phi$ and a map $\tilde{\phi}$ such that
  $\tilde{\psi}\circ f = g\circ \tilde{\phi}$. Because $f$ and $g$ are holomorphic, it follows
  that $\tilde{\phi}$ is also quasiconformal. 
\end{remark}
\begin{remark}[Remark 2]
 The set $\MfR$ can naturally be given the structure of a $q+2$-dimensional real-analytic
  manifold, where $q=\# S(f)$, as we discuss in Section \ref{sec:parameter}. For now,
  we only consider $\MfR$ as a set of entire functions. 
\end{remark}

Note that the maps $\phi$ and $\psi$ might not be uniquely determined.
When we speak of two real-topologically equivalent functions, we always implicitly
assume that a specific choice of $\phi$ and $\psi$ is given. Another way of saying this is
that we \emph{mark} the singular values and the critical points.

One important  consequence of $f,g$ being real-topologically equivalent
is that  if $c$ is a critical point of $f$ then $\phi(c)$ is a critical point of $g$
of the {\em same order}. 

\subsection*{Several notions of conjugacy}

Let $f, g\in \SR$ be real-topologically equivalent, with a suitable choice of
$\phi$ and $\psi$ as above. 
We will now discuss a number of different important notions of
conjugacies: \emph{combinatorial},
\emph{topological}, \emph{quasiconformal} and \emph{conformal}.

First we modify the definition of combinatorial conjugacy on the real line 
 (recall Definition~\ref{def:Rcomb-equiv}). The point of this
 modification is that, when we look at functions in the complex plane,
 we should restrict to those that are real-topologically equivalent. 
 Given such a real-topological equivalence, represented by maps
 $\phi$ and $\psi$, we have a natural correspondence between the
 critical points of $f$ and $g$ (via $\phi$) and the singular values
 of $f$ and $g$ (via $\psi$). Our combinatorial conjugacy should respect
 this information; i.e.\ map corresponding critical points and singular
 values to each other. Furthermore, if we wish to relate our maps $f$ and $g$ also
  in the complex plane, then we must consider not only
 the behaviour of points in $\SR(f)$, but include 
 \emph{all} singular values in the definition.

\begin{defn}[Combinatorial conjugacy for maps in $\SR$]
 \label{defn:combinatorialconjugacyinC}
Two functions $f,g\in\SR$ are called 
 \emph{combinatorially conjugate (in $\C$)} 
 if
 they are real-topologically equivalent, 
  say
 $\psi\circ f = g\circ\phi$, and 
 there exists an order-preserving
 bijection
 $$h\colon \Part(f) \cup O^+_f(S(f)) \to \Part(g) \cup O^+_f(S(g))$$
such that
\begin{enumerate}
 \item $h\circ f = g\circ h$, 
 \item $h|_{\Crit_{\R}(f)} = \phi|_{\Crit_{\R}(f)}$,
 \item $h|_{S(f)} = \psi|_{S(f)}$ and
 \item $h$ maps each nonrepelling periodic
    point to a nonrepelling periodic point of the same
   type (i.e. hyperbolic to hyperbolic, and parabolic to parabolic). 
\end{enumerate}
\end{defn}

The reason we say that $f$ and $g$ are combinatorially conjugate in $\C$
(rather than combinatorially conjugate on the real line) is that
the assumption that $f,g\in\SR$ are real-topologically equivalent
implies that $f,g$ are topologically conjugate on the complex plane
whenever the combinatorial conjugacy $h\colon \R\to \R$ is quasisymmetric, see
Theorem~\ref{thm:pullback}.

\begin{prop}[Combinatorial conjugacy in $\C$ implies topological conjugacy in $\R$] \label{prop:comb-top-conj}
If  $f,g\in\SR$ are combinatorially conjugate in $\C$, then these maps
are combinatorially conjugate on the real line (in the sense
of Definition~\ref{def:Rcomb-equiv}) 
and therefore topologically conjugate on the real line. Furthermore,
the topological conjugacy can be chosen to agree with the combinatorial
conjugacy from Definition \ref{defn:combinatorialconjugacyinC}.
\end{prop}
\begin{proof} Property (b) and the fact that $f,g$ are real-topologically equivalent 
imply that $h$ sends critical points of $f$ to critical points of $g$
of the same order. Also, the condition on asymptotic values is automatically satisfied: if
$\lim_{x\to\sigma\infty} f(x)=a$, then 
  \[ \lim_{x\to\sigma\infty}g(x)=\lim_{x\to\sigma\infty}g(\phi(x)) = \psi\bigl(\lim_{x\to\sigma\infty}f(x)\bigr)= \psi(a). \]
  The proposition therefore follows from 
Lemma~\ref{lem:comb-conj}. We note that, a priori, the
 topological conjugacy provided by this lemma is an extension of
 the real combinatorial conjugacy, which in general is a 
 \emph{restriction} of our original map $h$.
 However,
 the extension will automatically agree with our original map on points
 that do not belong to attracting or parabolic basins
 (due to absence of wandering intervals),
 and can easily be arranged to respect the finitely many remaining
 orbits.
\end{proof}

Combinatorial conjugacy can also be expressed alternatively in terms of \emph{kneading sequences},  which is an idea that we use later.

\begin{defn}[Topological and QC conjugacy]
Two maps $f,g\in\SR$ are called \emph{real-topologically conjugate} if
  there is a homeomorphism $\theta\in\HomeoR$ 
such that 
  $\theta\circ f = g\circ \theta$ on $\C$. 
  (The prefix ``real'' in this notation is to 
  express that $\theta$ preserves the real line.)

 If this homeomorphism $\theta$ is quasiconformal, we say that
  $f$ and $g$ are \emph{real-quasi\-con\-formally conjugate}. 
\end{defn}

Finally, let us turn to a notion of conjugacy on escaping sets.
 Recall that $\IR(f) = \{x\in \R\colon |f^n(x)|\to \infty\}$. 

\begin{defn}[Escaping conjugacy]
\label{def:escapingconj}
 Let $f,g\in\SR$ be real-topologically equivalent. We say that $f$ and $g$ are
  \emph{escaping conjugate} if there is an order-preserving homeomorphism
    $j: \IR(f)\to \IR(g)$ such that: 
  \begin{enumerate}
    \item $j\circ f = g\circ j$ on $\IR(f)$;
    \item $j$ agrees with $\phi$ on $\Crit_{\R}(f)\cap \IR(f)$ and with $\psi$ on
         $S(f)\cap \IR(f)$, and 
    \item  for every closed and forward-invariant set $K\subset \IR(f)$, $j|_K$ extends to a quasisymmetric
   homeomorphism of the real line.
  \end{enumerate}
\end{defn}

The article \cite{boettcher} provides a simple way of 
 encoding when two maps are
escaping conjugate. We discuss this below. For now, we only need the following fact.

\begin{thm}[Escaping rigidity]\label{thm:boettcher}
 If $f,g\in\SR$ are real-topologically conjugate, then they are escaping conjugate. 
\end{thm} 
\begin{proof}
  Let $h$ be the real-topological conjugacy between $f$ and $g$; we set $j=h$.
  The first two conditions in the definition of escaping conjugacy are trivially 
   satisfied (recall that we can take $\phi=\psi=h$ in the definition of topological
   equivalence). So let $K\subset \IR(f)$ be a closed and forward-invariant
   set; we must show that $h|_K$ extends to a quasisymmetric homeomorphism of the
   real line. This is stated explicitly in \cite[Theorem 1.3]{boettcher} for the case where
  the set $K$ is the union of
  finitely many escaping orbits (which is in fact
  sufficient for the purposes of this paper). 

 In general, we can deduce this claim from the results of 
    \cite{boettcher} as follows. First observe that, for every $R>0$ there is 
   $n_0\in\N$ such that $f^n(K)\cap [-R,R] = \emptyset$ for all $n\geq n_0$. 
   Indeed, by the disussion following Lemma \ref{lem:ahlfors}, 
   we may assume that $R$ is sufficiently large to ensure that
   $|f^n(x)|\geq R$ whenever $|f^2(x)|\geq R$. Hence we need to prove the claim only 
   for the set $K\cap [-R,R]$, where it is trivial by compactness.

Now \cite[Theorem 1.1.\ and Corollary 4.2]{boettcher} imply that
  $h|_{f^n(K)}$ 
   extends to a real-quasi\-conformal homeomorphism $\psi$. We may choose $\psi$ in such
  a way that it agrees with $h$ on the set of singular values of $f^n$. Hence $\psi$ and $h$
  are isotopic relative to $S(f^n)$, and we may lift the homotopy to obtain a map
   $\phi$ such that $\psi\circ f^n = g^n\circ \phi$. By construction,
   the real-quasiconformal map $\phi$ extends
   $h|_K$, as desired.
\end{proof}

\subsection*{Promoting conjugacies: the pullback argument} The following
  is a version of a well-known argument of promoting combinatorial conjugacies
  to quasiconformal ones, provided that one has control on the 
  postsingular set. 

\begin{thm}[The pullback argument]\label{thm:pullback}
Suppose that $f,g\in\SR$ are combinatorially conjugate (in $\C$) and that the
 combinatorial conjugacy $h$ extends to a quasisymmetric homeomorphism
 $h:\R\to\R$. 

 Then $f$ and $g$ are real-quasiconformally conjugate, where the conjugacy $\theta$ can
   be chosen to agree with $h$ on 
  $\Part(f)\cup O^+_f(S(f))$.
\end{thm}
\begin{proof}
 Since the map $h$ is quasisymmetric, it extends to a real-quasiconformal map
 $\theta_0:\C\to\C$. 
 Let $\phi$ and $\psi$ be the maps from the definition of 
 real-topological equivalence. 
 By the definition of a combinatorial conjugacy, the map $\theta_0$
 is isotopic to $\psi$
 relative $S(f)$. 

Furthermore, it
 follows from the assumption that $f$ and $g$ are combinatorially conjugate, or alternatively from the quasisymmetry of $h$, that every attracting cycle
 of $f$ maps to an attracting cycle of $g$, and every parabolic cycle of
 $f$ maps to a parabolic cycle of $g$ under $h$. 

Also note that, in the class $\SR$, every attracting direction of a 
 parabolic point must be aligned with the real axis, so
 there are only three possibilities
 for parabolic points: a parabolic point with one fixed attracting petal
 (corresponding to a saddle-node $z\mapsto z+z^2$)
    a parabolic point with two fixed attracting petals 
  (as for $z\mapsto z-z^3$), or one with
  a $2$-cycle of attracting petals (corresponding to a 
  fixed point with eigenvalue $-1$
  as in the period-doubling bifurcation). Since each attracting petal must contain some critical point, the combinatorial
  conjugacy must map each parabolic point to one of the same type. 

It is then easy to see that we can choose the map $\theta_0$ in such
 a way that $\theta_0$ is a conjugacy between $f$ and $g$ in some
 linearizing neighbourhood or attracting petal for every attracting 
 periodic point or parabolic attracting direction. This can be done as in Section  
 5 of   \cite{CS}.

By the covering homotopy theorem, we can find a map 
 $\theta_1$, isotopic to $\phi$ relative $f^{-1}(S(f))$, such that
 $\theta_0\circ f = g\circ \theta_1$. Here we use that $\phi$ agrees
 with $h$ on $\P(f)$ and that $h$ maps critical points of $f$ to 
 critical points of $g$ of the same order.
    Since $\phi$ preserves the real line,
 and $f$ and $g$ are real, we also get that 
 $\theta_1(\R)=\R$. 

We claim that
 $\theta_1$ agrees with the original map $h$ on the postsingular set.
 Indeed, let $v\in \P(f)$. By construction,
    \[ g(\theta_1(v)) = \theta_0(f(v)) = h(f(v)) = g(h(v)), \]
  so $\theta_1(v)$ and $h(v)$ both have the same image. Since
  $\theta_1 = \phi_1 = h$ on the set of critical points of $f$, 
  we see that $\theta_1(v)$ and $h(v)$ belong to the same
  interval of $\R\setminus\Crit(g)$, and since $g$ is injective on each
  of these intervals, we have $\theta_1(v) = h(v)$ as desired. 

In particular, $\theta_1$ is also isotopic to $\psi$, and we can repeat the
 above procedure to obtain maps $\theta_j$ with
    \[ \theta_j\circ f = g\circ \theta_{j+1}, \]
 and such that $\theta_{j}$ is isotopic to $\psi$ relative
to the postsingular set and isotopic to $\phi$ relative 
  $\Crit_{\R}(f)$. 

 Note that the maps $\theta_j$ and $\theta_{j+1}$ agree on the $j$-th
  preimages of the union of the postsingular set with the originally chosen
  linearizing neighbourhoods and parabolic petals. 
  Also note that their maximal
  dilatation does not increase with $j$. Hence $\theta_j$ converges
  to a suitable quasiconformal function $h$, which is the desired
  conjugacy.    
\end{proof}

\section{Rigidity}
\label{sec:rigidity}
In this section, we establish our main rigidity theorem.

\begin{thm}[From combinatorial to quasiconformal conjugacy]
\label{thm:realqcconj}
Let $f,g\in\SR$ be combinatorially conjugate (in $\C$) 
 and escaping conjugate.
Then $f$ and $g$ are real-quasiconformally conjugate
 (and 
  the conjugacy can be chosen to be an extension of the original combinatorial
  conjugacy).
  \end{thm}
\begin{remark}
 For functions that are bounded on the real line, that have no asymptotic values and for which all critical points are real, the statement
  simplifies as follows: If $f$ and $g$ are real-topologically combinatorially equivalent and topologically conjugate on the real line, then they
   are real-quasiconformally conjugate in the complex plane.
\end{remark}
\begin{proof}
Let 
 $$h\colon \Part(f) \cup O^+_f(S(f)) \to \Part(g)\cup O^+_f(S(g))$$
 be the combinatorial conjugacy between $f$ and $g$. 
 We write $\dom(h) := \Part(f)\cup O^+_f(S(f))$ for the domain of $h$.
 
\begin{claim}
 There exists a quasisymmetric extension of  $h$ to $h:\R\to\R$.
\end{claim}
\begin{subproof}
 Theorem~\ref{thm:realrigidity} asserts that  there exists a compact interval $J\subset\R$ (possibly empty or consisting of only one point)    with the following properties.
 \begin{enumerate}
  \item For every $x\in\R$, either $x\in \IR(f)$ or 
    $f^j(x)\in J$ for all $j\geq 2$.
  \item The set of points $z\in\C$ whose $\omega$-limit set is contained in
   $J$ and which do not belong to an attracting or parabolic basin does not
   support any invariant line fields.
  \item If $g\in\BR$ is  combinatorially conjugate on the real line to
     $f$, then there exists an interval $\tilde J$ with corresponding properties and an
     order-preserving
     quasisymmetric homeomorphism $h_1\colon \R\to \R$ which maps $J$ onto $\tilde J$
     and so that $h_1\circ f=g\circ h_1$ on $J$ and such that $h_1=h$ on
     $\dom(h)\cap J$.
 \end{enumerate}
 Let $I_+$ and $I_-$ denote the two components of $\R\setminus J$, and 
  let $\tilde{I}_+$ and $\tilde{I}_-$ be the corresponding components of $\R\setminus \tilde{J}$,
  i.e. $\tilde{I}_{\sigma} = h_1(I_{\sigma})$. Fix $\sigma\in\{+,-\}$.

\begin{claim}[Subclaim] 
 The restriction of $h$ to $\dom(h)\cap I_{\sigma}$ can
 be extended to an order-preserving quasisymmetric
  homeomorphism $h_{\sigma}:\R\to\R$. 
\end{claim}
  
  To see this, note that $\dom(h)\cap I_{\sigma}$ is a closed and discrete subset
  of the real line. We distinguish three cases:
 \begin{itemize}
  \item If $\dom(h)\cap I_{\sigma}$ is finite, then the subclaim is trivial.
  \item If $I_{\sigma}$ contains infinitely many postsingular points, then $|f|$ is unbounded
   as $x\to\sigma\infty$, and in particular $\dom(h)\cap I_{\sigma}$ consists of finitely many 
   escaping singular orbits (possibly together with finitely many additional points). The
   subclaim follows from the assumption that $f$ and $g$ are escaping conjugate. 
  \item If $\dom(h)\cap I_{\sigma}$ is infinite but contains only finitely many
    postsingular points, it must contain infinitely many critical points. The subclaim follows
    from the fact, remarked after Definition \ref{def:MfR},  that the restriction of
    $\phi$ to the set of critical points of $f$ extends to a quasisymmetric homeomorphism. 
    This completes the proof of the subclaim.
 \end{itemize}
 Because $\dom(h)\cap I_{\sigma}$ is a closed set, we can construct the desired extension
  $h:\R\to\R$ by interpolating between $h_-$, $h_1$ and $h_+$. 
 (E.g., $h$ agrees with 
  $h_1$ on $J$ and with each $h_{\sigma}$ on a closed subinterval of $I_{\sigma}$ which contains
  $\dom(h)\cap I_{\sigma}$ and is linear on the complement of these intervals.) This completes the proof of the claim.
\end{subproof}

  The assertion in the theorem now follows from the pullback argument (Theorem~\ref{thm:pullback}). 
 \end{proof}

\begin{proof}[Proof of Theorem~\ref{thm:qcrigidity}]
 Suppose that $f,g\in\SR$ are topologically conjugate by a conjugacy $h$
  that preserves the real axis. We may assume that $h$ commutes with complex conjugation
  (otherwise, replace $h$ on the lower half-plane by the map $h(z):=\overline{h(\bar{z})}$). 
   Hence either $f$ and $g$ are real-topologically
  conjugate (if $h|_{\R}$ is order-preserving) or
  $f$ and $\tilde{g}(z) := -g(-z)$ are real-topologically conjugate 
  (otherwise). So the claim follows from 
  the previous theorem and Theorem~\ref{thm:boettcher}.
\end{proof}

\section{Absence of Invariant Line Fields}
\label{sec:absenceline}

\subsection*{Absence of invariant line fields in $\SR$}
In this section, we are concerned with showing that the functions
 $f\in\SR$ we consider do not support any invariant line fields on their
 Julia sets. (Recall the definitions from Section \ref{sec:prep}.) As
 mentioned in the introduction, we will do so by decomposing 
 the Julia set in a number of dynamically distinct sets and treat each
 separately.

 So let $f\in \SR$ and define
 \begin{align*}
   \P_B(f) &:= \{z\in\P(f): O^+(z)\text{ is bounded}\} \quad\text{and}\\
   \P_I(f) &:= \P(f)\cap \IR(f) = \P(f)\setminus \P_B(f).
 \end{align*} 
We consider the following subsets of the complex plane:
\begin{enumerate}
 \item The radial Julia set $J_r(f)$ (by definition this is the set of all points $z\in J(f)$ 
   with the following property:     there is some $\delta>0$ such that, for infinitely many $n\in\N$, the 
   disk $\D_{\delta}^{\#}(f^n(z))$ can be pulled back univalently along the
   orbit of $z$).
 \item The escaping
   set $I(f)=\{z\in\C: |f^n(z)|\to \infty\mbox{ as } n\to \infty\}$.
 \item The set $L_B(f)$ of points $z\in J(f) \setminus J_r(f)$ with $\dist(f^n(z),\P_B(f))\to 0$.
 \item The set $L_I(f)$ of points $z\in J(f)\setminus (J_r(f)\cup I(f))$ 
    with $\dist^{\#}(f^n(z),\P_I(f))\to 0$. 
\end{enumerate}

\begin{lem}[Partition of the Julia set]\label{lem:decomp}
For any $f\in\SR$, we have $J(f) = J_r(f) \cup I(f) \cup L_B(f) \cup L_I(f)$.
\end{lem}
\begin{proof}
Any point with $\limsup \dist^{\#}(f^n(z),\P(f)) > 0$ belongs to $J_r(f)$. So
 it remains to show that an orbit cannot accumulate both on bounded and
 on escaping singular orbits. 

This follows from continuity of $f$. Indeed, consider the spherical
 distance 
 $\delta := \dist^{\#}(\P_B(f) , \P_I(f)\cup\{\infty\})$. Since the set of
 singular values is finite, the sets $\P_B(f)$ and $\P_I(f)\cup\{\infty\}$ are
 both compact  subsets of $\bar \C$, hence we have 
 $\delta>0$. 

 Then there exists 
 $\eps\in (0,\delta/2)$ such that $\dist^{\#}(\P_B(f),f(z))< \delta/2$ 
 for any point $z\in\C$ with
  $\dist(\P_B(f),z) < \eps$. If $z\in \C\setminus J_r(f)$, then
 $\dist^{\#}(f^n(z),\P(f)) < \eps$ for sufficiently large $n$. We then have either
 $\dist^{\#}(f^n(z),\P_B(f)) \geq \eps$ for all such $n$, or 
 $\dist^{\#}(f^n(z),\P_B(f)) < \eps$ for all sufficiently large $n$. In the
 former case, we must have $z\in I(f)\cup L_I(f)$, while in the latter,
 $z\in L_B(f)$. \end{proof}

\begin{thm}[No invariant line fields on radial, escaping and bounded orbits]\label{thm:nolinefieldJr}
 Suppose that $f\in \SR$. 
Then the sets $J_r(f)$, $I(f)$ and $L_B(f)$ support no invariant line fields. Furthermore,
  if $f$ has bounded criticality, then the set $L_I(f)$ has zero Lebesgue measure, and hence
  also does not support invariant line fields.
  \end{thm}
\begin{proof}
The set $J_r(f)$ (of any transcendental meromorphic function) 
 does not support any invariant line field by \cite[Corollary 7.1]{linefields}. (Compare also
  \cite{conicalrigidity}.)
The set $I(f)$ does not support any invariant line fields; in fact,
  this is true for all transcendental entire functions for which 
  $S(f)$ is bounded
  \cite{boettcher}.
The set $L_B(f)$ supports no invariant line fields by 
 Theorem \ref{thm:realrigidity} above. 

Finally, let $z\in L_I(f)$, and let $v\in \P_I(f)$ be a limit point of the orbit of $z$; say
   $f^{n_i}(z)\to v$. If $D$ is a small disk
   around $v$, and $D_i$ is the component of $f^{-n_i}(D)$ containing $z$, then the
   assumption of bounded criticality implies that the restrictions
   $f^{n_i}:D_i\to D$ are proper maps of bounded degree (independent of $i$). By
   Lemma \cite[Lemma 3.6]{linefields}, the set of such points has Lebesgue measure
   zero, as claimed. (Compare also the proof of Claim 1 in the proof of Theorem 
   \ref{thm:nolinefieldsLI} below.)
\end{proof} 

\subsection*{Absence of invariant line fields on points asymptotic to singular orbits}
We now come to the main new result of this section.

\begin{thm}[Absence of invariant line fields on $L_I(f)$]
\label{thm:nolinefieldsLI}
Suppose that $f\in\SR$ satisfies the sector condition
 (Definition \ref{def:sector}).
 Then $L_I(f)$ supports no invariant line fields. 
\end{thm}
\begin{proof}
Suppose by contradiction that $L_I(f)$ supports a measurable 
invariant line field $\mu$.
As mentioned, this means that there exists a set $A\subset L_I(f)$ of 
positive Lebesgue
measure so that $A\ni z\mapsto \mu(z)$ is a measurable choice of a (real) 
line 
through $z$ (i.e. it is a measurable 
map from $A$ into the projective plane). 

The rough idea of the proof is as follows. First of all, we let $z$ be
a 
point of continuity of the line field $\mu$, and will observe that
(unless $z$ belongs to a set of measure zero)
its orbit
must accumulate at some point $v\in \P_I(f)$, passing
either through 
transcendental singularities or through neighbourhoods of critical points
of high degree. This will allow us to conclude that $v$ has circular
neighbourhoods
in which the line field $\mu$ looks almost like a
\emph{radial line field}   $\theta(z)= \rho z/|z|$, 
 where $\rho\in \C$ with $|\rho|=1$. (See Figure
 \ref{fig:arnold}.) More precisely, we show:

\begin{claim}[Claim 1]
For almost every $z\in A$, the following holds. Let $v\in \P_I(f)$ be
 an accumulation point of the orbit of $z$. Then there exists a sequence
 $\delta_i\to 0$ of radii such that the rescalings 
   \[ \tilde{\mu}_i(z) := \mu(\delta_i z + v), \quad z\in \D = B_1(0), \]
 converge to a radial line field $\theta(z)=\rho z/|z|$ on $\D$. 
 (Here, convergence means that for any $\epsilon>0$  
  there exists a set $X_\epsilon\subset \D$ so that the Lebesgue
 measure of $\D\setminus X_\epsilon$ is less than $\epsilon$
 and so that  $\tilde \mu_i$ is defined on $X_\epsilon$
 and converges uniformly to $\theta$ on $X_\epsilon$.)
\end{claim}

\begin{figure}[ht]
\begin{center}
\includegraphics[width=4cm]{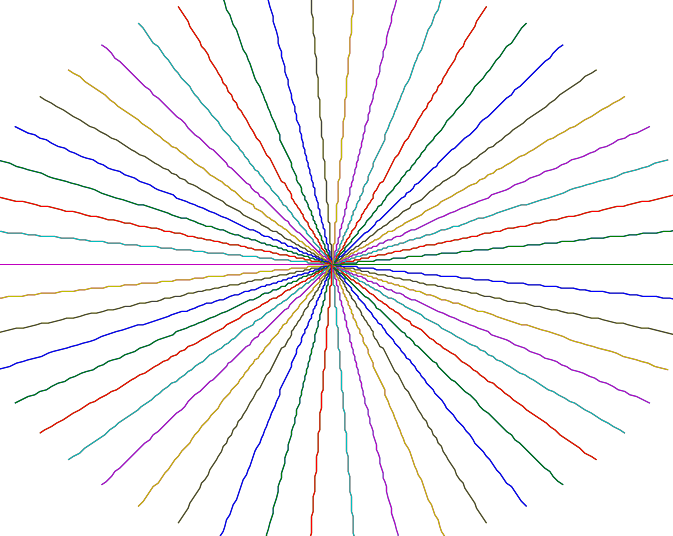}
\end{center}
\caption{Near the value $v$ the line field $\mu$ is almost radial.}
\label{fig:prong}
\end{figure}

Once this claim is established, 
 we take
 forward iterates of the disk $B_{\delta_i}(v)$, 
 until it stretches many times over 
 some large annulus. Given what we know about the line field on $A_i$, we
 can derive a contradiction.

To make this idea more precise, we use the
\emph{logarithmic change of variable} \cite[Section 2]{alexmisha}.
If $v$ is a limit point as in Claim 1, then 
$|f^n(v)|\to \infty$ as $n\to \infty$. Because the set of singular values
of $f$ is finite (and hence bounded), there is 
$\sigma\in \{+,-\}$
so that $f^n(v)\to \sigma \infty$ or so that 
$(-1)^n f^n(v)\to \sigma \infty$.
In the first situation
$\lim_{x\to \sigma \infty} f(x)=\sigma \infty$ and in the second case
$\lim_{x\to \sigma \infty} f(x)=-\sigma \infty$. To fix our ideas, let us
suppose that we are in the former case; the arguments in the latter are 
analogous. (Note, however,
that the sector condition is not preserved under iteration,
so we cannot simply reduce the second case to the first by considering
$f^2$ instead of $f$.)
Thus we assume that  $\lim_{x\to +\infty} f(x)=+\infty$ and that
there exists a point $v$ as in Claim 1 such that 
$f^n(v)\to +\infty$. (In particular, we have $+\in\Sigma$, where
$\Sigma$ is the set from the sector condition.)

Choose $M>0$ large enough such that $M>|f(0)|$, such that
$f(x)>x$ for $x\geq M$ and such that
\[ E(M):=\{z\in \C; |z|>M\} \]
 contains no
singular values of $f$. Let $V$ be the component of
$f^{-1}(E(M))$ that contains $[M,\infty)$. 

Since $E(M)$ contains no singular values,  
 $f\colon V\to E(M)$ is a covering map.
 Since $f$ is transcendental, $V$ is simply-connected and 
 $f\colon V\to E(M)$ is a universal covering. Set $r := \log M$, 
 $\HH_{r} := \{z\in\C: \re z > r\}$ and
 let $W$ be the component of $\exp^{-1}(V)$ that contains
 $[r,\infty)$. Because 
 $\exp:\HH_{r}\to E(M)$ is also a universal covering map, and
 $\exp:W\to V$ is univalent, there is a 
 conformal isomorphism $F\colon W\to \HH_{r}$ such that 
 $\exp\circ F=f\circ \exp$ and $F(W\cap\R)\subset \R$.
\[ \xymatrixcolsep{5pc} 
\xymatrix{ W \ar[r]^F\ar[d]_{\exp}& \HH_r\ar[d]^{\exp}   \\ 
   V\ar[r]^f  & E(M)  }
\]

It is well-known that the map $F$ is strongly expanding, see  equation (\ref{eqn:eremenkolyubich}) below, 
 and we will use this,
 together with the sector condition, to blow up the almost radial line field
 from Claim 1 to a large scale (in logarithmic coordinates). More precisely,
 we use the following.

\begin{claim}[Claim 2]
There exist constants $r_1>r$ and $c<1$ such that, for every $K\ge c>0$, there
 is $\delta_0=\delta_0(K)$ with the following property. 

Let $w\geq r_1$ and
  $\delta \leq \delta_0$. 
 Then there exist
 $\tilde{\delta}\leq\delta$ with $\tilde{\delta}\geq c\cdot \delta/K$ and
  a number $n\geq 0$ such that
 $F^n$ is defined and univalent on $B_{\tilde{\delta}}(w)$ and
   \[ F^n(B_{\tilde{\delta}}(w))\supset B_{K}(F^n(w)). \]
\end{claim}

To show how these two claims, together, yield the theorem,
 let $v\in \P_I(f)$ be a point as in Claim 1 such that
 $f^n(v)\to +\infty$. By passing to a forward iterate, if necessary,
 we can assume that $f^n(v)>e^{r_1}$ for all $n\geq 0$, where $r_1$
 is as in Claim 2. So $f^n(v)\in E(M)$ for all $n\ge 0$. Set
 $w:=\log v\in \R$. 

Let $\mu'$ be the line field on 
 $\HH_{r}$ defined by pulling back $\mu$ under $\exp\colon \HH_r\to E(M)$. Then
 $\mu'$ is $2\pi i$-periodic by definition. It follows from
 Claim 1 that there is a sequence 
 $\delta_i\to 0$ of radii such that the rescalings of $\mu'$ on
 the disks $B_{\delta_i}(w)$ converge to a radial line field.  

If we let $K_i$ be a sequence that tends to infinity sufficiently slowly,
 then for any choice of $\eps_i$ between 
 $4\pi \cdot c\cdot\delta_i/K_i^2$ and 
  $\delta_i$, the rescalings of $\mu'$ on the
 disks $B_{\eps_i}(w)$ will also converge to a radial line field. 
 (Here $c$ is the constant from Claim 2.) We may also assume that 
 $K_i>4\pi$ and $\delta_0(K_i)>\delta_i$ for all $i$. 

Now we apply Claim 2 to obtain numbers $\tilde{\delta}_i$ with
 $c\cdot \delta_i/K_i \leq \tilde{\delta}_i \leq \delta_i$ as well as numbers
 $n_i$ such that $F^{n_i}$ is defined and univalent on 
 $B_{\tilde{\delta}_i}(w)$ and covers 
 $B_{K_i}(F^{n_i}(w))$.
 If we set $\eps_i := 4\pi\cdot \tilde{\delta}_i/K_i$, then 
 $F^{n_i}(B_{\eps_i}(w))\supset B_{4\pi}(F^{n_i}(w))$. To see this, 
 apply the Schwarz Lemma to the branch of $F^{-n_i}|B_{K_i}(F^{n_i}(w))$
 mapping into $B_{\tilde{\delta}_i}(w)$.

Since $F^{n_i}(B_{\eps_i}(w))$ can be much larger than $B_{4\pi}(F^{n_i}(w))$,
we define $\kappa_i>0$ to be the largest integer so that $\phi_i(\D)\supset 
B_{4\pi}(0)$ where 
  \[ \phi_i:\D\to\C;\quad
       \zeta \mapsto \frac{F^{n_i}(w+\eps_i\cdot \zeta) - F^{n_i}(w)}{\kappa_i}.\]
Passing to a subsequence again if necessary, we can assume that the $\phi_i$ converge uniformly to a non-constant linear map. 
  (Recall that each $\phi_i$ extends to a conformal map on a disk whose radius tends to $\infty$ as $i\to\infty$.) 
 Now consider the sequence of line fields on the disk $D=B_{4\pi}(0)$ obtained by first rescaling the line field $\mu'$ on the disk
   $D:=B_{\eps_i}(w)$ as above, 
   and then pushing forward under the map $\phi_i$. By construction, these push-forwards converge to the radial line field
   on $D$. On the other hand, by invariance of $\mu'$, these push-forwards 
   are each obtained from $\mu'$ by a translation 
  and a rescaling by a factor of $1/\kappa_i$, and hence are all $2\pi i$-periodic. (Here we use  
   that $\kappa_i$ is an integer, and consequently $\mu'$ is $2\pi \kappa_i$-periodic).
   But then we obtain that
  the radial line field on the disk $D$ is also
  $2\pi i$-periodic, which is absurd. 
\smallskip 

It remains to establish Claims 1 and 2. To prove the former, let 
 $z$ be a Lebesgue density point of $A$ 
 which is also a point of continuity of $\mu$. This means that for 
 each $\epsilon>0$ there exists $\delta>0$ and a fixed line 
 $\mu_0$ so that
 $\dens(A, B_{\delta}(z))\geq 1-\epsilon$ and so that 
 $|\{z\in B_\delta(z)\cap A; |\mu(z)-\mu_0|\le \epsilon\}|/|B_\delta(z)|\ge 1-\epsilon$. 
Here ,
  \[ \dens(A,B) := \frac{\meas(A\cap B)}{\meas(B)} \]
 denotes the density of $A$ in $B$ and 
$|\mu(z)-\mu_0|$ denotes the angle between the lines $\mu(z)$
 and $\mu_0$.

Let $v\in \P_I(f)$ be a limit point of the orbit of $z$; say
  $f^{n_i}(z)\to v$. 
 Since the set of singular
 values of $f$ is finite, we can take $r>0$ so small that the set 
 $U:=B_r(v)$ does not intersect
 $\P(f)\setminus \{v\}$. We may assume that $f^{n_i}(z)\in U$ for all $i$. 
 Let $U_i$ be the component of $f^{-n_i}(U)$ that contains
 $z$. Let us also denote by 
 $U_i^*$ the component of $f^{-n_i}(U\setminus\{v\})$ contained in
 $U_i$. 

Then $U_i$ is simply connected, and since
 $z\in J(f)$, we have 
 \begin{equation}
   \dist(z,\partial U_i)\to 0. 
 \end{equation}
Furthermore,
 $f^{n_i}:U_i\to U$ is either a finite-to-one
 covering map of
 some degree $d_i<\infty$ (branched only over $v$) or 
 $f^{n_i}:U_i\to U\setminus\{v\}$ is a universal covering (of degree
 $d_i=\infty$). 
 Note that $U_i^*=U_i$ when $d_i=\infty$, whereas otherwise
 $U_i\setminus U_i^*$ consists of a single iterated preimage $v_i$ of $v$. 
 The set of points $z\in L_I(f)$ for which the sequence $d_i$ does not
 tend to infinity has Lebesgue measure zero by \cite[Lemma 3.6]{linefields}. So we may
 assume that $z$ was chosen such that $d_i\to\infty$. 

Let  $\HH=\{z\in \C; \re(z)>0\}$ denote the right half plane, and define
  \[ E: \HH \to U\setminus\{v\};\quad z\mapsto v + r\cdot e^{-z}. \]
Since $E$ is a universal covering, there exists a covering map 
 $\psi_i:\HH\to U_i^*$ with $f^{n_i}\circ \psi_i = E$. Since
 $f^{n_i}\colon U_i^*\to U\setminus \{v\}$ has degree $d_i$, 
  $\psi_i$ will be injective when restricted to any horizontal strip
 of height $2\pi d_i$. 

\[\xymatrixcolsep{5pc} 
 \xymatrix{ 
     U_i^*\ar[rd]^{f^{n_i}} &\HH\ar[l]_{\psi_i}\ar[d]^E \\ 
   & U\setminus\{v\}} \]

Let $\zeta_i$ be a 
 preimage of $f^{n_i}(z)$ under $E$, and define
   \[ R_i := \re \zeta_i. \]
 Since $f^{n_i}(z)\to v$, we have $R_i\to \infty$.

\begin{subproof}[Proof of Claim 1]
Let $\Delta_i < 2\pi d_i$ be a sequence 
 that tends to infinity
 sufficiently slowly (to be fixed below). 
 For simplicity let us also require that
 $\Delta_i$ is a multiple of $2\pi$. Consider the squares 
 \[ Q_i := \zeta_i + 
  \left[\frac{-\Delta_i}{2},\frac{\Delta_i}{2}\right]\times 
             \left[-i\frac{-\Delta_i}{2},i\frac{\Delta_i}{2}\right]
      \]
 with sides of length $\Delta_i$ and centre $\zeta_i$. Note that
 $\psi_i$ is injective on $Q_i$. Indeed, if 
   \[ S_i := \{a + ib: a> 0, |b-\im\zeta_i|< \pi d_i\} \]
 is the horizontal strip of
 height $2\pi d_i$ centered at $\zeta_i$, then $\psi_i$ is injective on
 $S_i$, as mentioned above. Furthermore, 
 if $\Delta_i$ grows sufficiently slowly, then 
 $\mod(S_i\setminus Q_i)\to\infty$, and hence
   \begin{equation} \label{eqn:bigmod}
     \mod(\psi_i(S_i)\setminus \psi_i(Q_i))\to\infty. \end{equation} 
 Let $\nu$ be the line field on $Q_i$ that is obtained by pulling back
  $\mu$ under $\psi_i$. Using (\ref{eqn:bigmod}), 
  the Koebe Distortion Theorem and that $z$ is a 
  point of continuity of the line field $\mu$, we see that $\nu$
  is an almost constant line field on $Q_i$.
 More precisely,  there is a sequence $\eta_i\to 0$ such that,
 for each $i$, there are a subset 
 $\hat Q_i$ of $Q_i$
 and a constant line field $\nu_0$ so that $\dens(\hat Q_i,Q_i)\ge 1-\eta_i$
 and $|\nu(z)-\nu_0|\le \eta_i$  for each $z\in \hat Q_i$. Moreover,
 if we decrease $\Delta_i$, then the bound for $\eta_i$ from the Koebe
 Theorem improves. This means that we may assume that $\Delta_i$ tends to
 infinity sufficiently slowly to ensure that
   \[ \Delta_i\cdot \eta_i \to 0. \]

Let us determine $\mu$ on $A_i=E(Q_i)=f^{n_i}(\psi_i(Q_i))$, using 
 the fact that $\mu|_{A_i} = E_*(\nu|_{Q_i})$.
 Note that
 $A_i$ is a round 
 annulus 
 centered around $v$ with $\mod(A_i)\to \infty$; let $r_i$ denote its
 outer radius. Also note that 
 $\theta=E_*(\nu_0|Q_i)$ is the radial line field 
 $z\mapsto \rho z/|z|$
 where $\rho\in \C$ with $|\rho|=1$ is constant; see Figure
 \ref{fig:prong}. Furthermore, 
   \[ \dens(E(Q_i\setminus \hat{Q}_i),A_i) 
          \leq  \dfrac{1}{1-e^{-\Delta_i}}
         \Delta_i\cdot \eta_i \to 0. \]
Indeed,  considering the map $E\colon Q_i\to E(Q_i)$  on a horizontal segment $L$,
and writing $T_i:=(Q_i\setminus \hat Q_i)\cap L$, 
we have
$$\dfrac{\mbox{length}(E(T_i))}{\mbox{length}E(Q_i\cap L)}
\le \dfrac{\mbox{length}(T_i)}{1-e^{-\Delta_i}} =\mbox{dens}(T_i,Q_i\cap L)\cdot \dfrac{\Delta_i}{1-e^{-\Delta_i}}.$$
 This completes the proof of Claim 1.
\end{subproof}

Now let us turn to the proof of Claim 2. We begin by reformulating the
sector condition in logarithmic coordinates.

\begin{claim}[Claim 3]
There exists $\epsilon_0\in (0,1)$ and $r_2>r+1$ so that 
$W\supset H(\epsilon_0,r_2)$  where 
\[H(\epsilon_0,r_2):=\{z\in \C; z=x+iy, x,y\in \R\mbox{ with } x> r_2\mbox{ and }|y|<\epsilon_0\} .\]
\end{claim}
\begin{subproof}
Note that $f$ satisfies the sector condition, and therefore $V$ contains
a sector of the form
\[ S(\theta,M_2):=\{z\in \C; z=x+iy, x,y\in \R \mbox{ with }|y|<\theta|x|, x\ge M_2\} \]
where $\theta>0$ and $M_2>0$ is some large number.
There exist $\epsilon_0\in (0,1)$ and $r_2>0$ so that
$S(\theta,r_2)\supset \exp(H(\epsilon_0,M_2))$
concluding the proof of the claim.
\end{subproof}

\begin{subproof}[Proof of Claim 2]
By
 \cite[Lemma 1]{alexmisha} (which is an application of Koebe's theorem), 
  the map $F$ is expanding: 
  \begin{equation} \label{eqn:eremenkolyubich}
      |F'(z)| \geq \frac{1}{4\pi}(\re F(z) - r) 
  \end{equation}
 for all $z\in W$. (Recall that $r=\log M$.)
 In particular, if $r_1>r_2+\eps_0$ 
 is sufficiently large, then
 for every $x\geq r_2$ and all $j\geq 0$, there exists a branch of
 $F^{-j}$ that takes $F^j(x)$ to $x$ and is defined on the disk
 of radius $3\eps_0$ around $F^j(x)$. 

Now let $w\geq r_1$, $K>0$ and $\delta>0$. We set $w_j := F^j(w)$, 
 $D := B_{\delta}(w)$ and $D_j := F^j(D)$. 
 Let $m\geq 0$
 be minimal such that $D_m$ is not contained in the strip
 $H(\eps_0/4,r_2)$. Then $F^m:D\to D_m$ is a conformal isomorphism.

 We claim that there is a universal constant
 $C$ such that 
   \begin{equation} \label{eqn:koebe}
\frac{\max_{\zeta\in\partial D_m}|\zeta-w_m|}{\min_{\zeta\in\partial D_m}|\zeta-w_m|}
               \leq C. \end{equation}
 This is trivial if $m=0$. Otherwise, let $\phi$ be the branch of
  $F^{-(m-1)}$ that takes $w_{m-1}$ to $w$ and is defined on the disk of
  radius $3\eps_0$. By definition of $m$, there is some point 
  $\zeta\in\partial D_{m-1}$ with $|\zeta - w_{m-1}|\leq \eps_0/4$. If
  $\omega\in\partial B_{\eps_0/2}(w_{m-1})$, we see by the Koebe distortion
  Theorem~\ref{thm:koebe} that
   \[ |\phi(\omega) - w| \geq 
         \frac{\frac{1}{6}}{(1+\frac{1}{6})^2}\cdot
         \frac{(1-\frac{1}{12})^2}{\frac{1}{12}}\cdot 
         |\phi(\zeta)-w| > |\phi(\zeta)-w| = \delta. \]
  Thus it follows that $D_{m-1}\subset B_{\eps_0/2}(w_{m-1})$, and
   (\ref{eqn:koebe}) follows from the Koebe Distortion Theorem (using
   the fact that $F$ is univalent on the disk $B_{\eps_0}(w_{m-1})$). 

 If $B_{K}(w_m)\subset D_{m}$, then we set $n:= m$ and are done. Otherwise,
  define $R_1 := \max_{\zeta\in\partial D_m}|\zeta-w_m|$ and
  $R_2 := \min_{\zeta\in\partial D_m}|\zeta-w_m|$, so that
  $R_1/R_2\leq C$ and $R_2<K$. We set
    \[ \tilde{\delta} := \frac{\delta\cdot \eps_0}{R_1}, \]
  $\tilde{D}:= D_{\tilde{\delta}}(w)$ and $\tilde{D}_m := F^m(\tilde{D})$. 
  Note that
     \[ \tilde{\delta} > \frac{\eps_0}{C}\cdot \frac{\delta}{K} .\]
   To prove Claim 2, we define $c:= \epsilon_0/C$ and need to
   check that    $F^{m+1}(\tilde D)\supset B_K(F^{m+1}(w))$. 
   To see this, notice that $F^m(B_\delta(w))\supset B_{R_1}(F^m(w))$.
 Hence by Schwarz and by the choice of $\tilde \delta$, 
  $\tilde{D}_m=F^m(B_{\tilde \delta}(w))\subset H(\eps_0,r_2)$.
  It follows that $F^{m+1}$ is 
  defined and univalent on $\tilde{D}$. Furthermore, by
  Koebe's theorem, $\tilde{D}_m$ contains the disk $B_{C_1\cdot \eps_0}(w_m)$
  for a universal constant $C_1$. It follows, again using Koebe's
  theorem and the estimate (\ref{eqn:eremenkolyubich}) that 
    \[ F(\tilde{D}_m) \supset B_{K'}(w_{m+1}), \]
  where 
     \begin{align*}
          K' &= C_1\cdot \eps_0 \cdot |F'(w_m)|/4 \geq 
         \frac{C_1\cdot \eps_0}{16\pi}\cdot (\re w_{m+1}-r).
     \end{align*}
 Note that, as $\delta\to 0$, we have $m\to\infty$ and hence 
  $\re w_{m+1}\to\infty$. Thus we can choose $\delta_0$ sufficiently small
  that $\delta<\delta_0$ implies $K'\geq K$, which completes the proof.
\end{subproof}
\end{proof}

\section{Parameter spaces}
\label{sec:parameter}

Recall that, given $f\in\SR$, we denote by 
 $\MfR$ the set of functions real-topologically equivalent to $f$
  (Definition \ref{def:MfR}). 
 As we have already mentioned, this space can naturally be given the structure
 of a real-analytic manifold; this follows from 
 work of Eremenko and Lyubich
 \cite{alexmisha} (who treated the complex-analytic case). More precisely:

 \begin{prop}[Manifold structure]
  Let $f\in\SR$ and set $q:= \# S(f)$.
   Then the set $\MfR$ can be given the structure of a real-analytic manifold
   of dimension $q+2$ in such a way that:
  \begin{itemize}
   \item A sequence $f_n\in \MfR$ converges to $f$ in the manifold topology of $\MfR$ 
     if and only if there are sequences of homeomorphisms
     $\psi_n,\phi_n\in\HomeoR$
     converging to the
      identity as $n\to\infty$
     such that $f_n = \psi_n\circ f \circ \phi_n^{-1}$.
   \item The inclusion from $\MfR$ (as a real-analytic manifold) to the space
     of entire functions is real-analytic. 
  \end{itemize}

 In the following, we will always assume $\MfR$ to be equipped with this
  topology and real-analytic structure. If we wish to make the distinction, we will
  refer to this as the ``manifold topology'', and the induced topology from the space of
  entire functions as the ``locally uniform topology''. 
 \end{prop}

 The fact that the dimension of $\MfR$ is $q+2$ (rather than $q$) reflects the 
  fact that the group $\MobR$ of order-preserving real affine maps acts on $\MfR$ by
  conjugacy. We can quotient $\MfR$ by the action of this group:

 \begin{prop}[The quotient $\MfRt$]
  Let $\MfRt$ be the quotient of $\MfR$ by $\MobR$.
   Then $\MfRt$ is a real-analytic manifold of dimension
   $q=\# S(f)$, and the projection 
   $\pi:\MfR\to \MfRt$ is real-analytic. 
 \end{prop}

 To prove our main results, we shall first establish
  density of hyperbolicity in $\MfRt$ (provided $f$ satisfies the
  sector condition). The following fact then implies that the same
  is true under the (a priori) more general hypotheses given in the
  introduction.

 \begin{prop}[Continuous families and the manifold topology]%
 \label{prop:continuity}
 Let $n,k\in\N$ and suppose that $(f_t)_{t\in[-1,1]^k}$ is a continuous family of functions
  $f_t\in \SR$ such that $\# S(f_t)=n$ for all $t$. Then there exist
  continuous families $\phi_t,\psi_t\in\HomeoR$
  such that
    \[ f_t = \psi_t\circ f_0\circ \phi_t^{-1} \]
  for all $t\in [-1,1]^k$.

  In other words, $f_t\in M_{f_0}^{\R}$ for all $t\in [-1,1]^k$ and
    $f_t$ depends continuously on $t$ in the topology of $M_{f_0}^{\R}$. 
 \end{prop}

For completeness and future reference, we provide a proof of the preceding 
 propositions in Appendix \ref{app:parameter}. In fact, we will give a very explicit
 topological description of the spaces $\MfR$ and
 $\MfRt$.

\begin{cor}[Connected conjugacy classes]\label{cor:conjconnected2}
Let $f\in \SR$. Then the set of functions $g\in \SR$ that are 
 real-topologically conjugate
 to $f$ is a connected subset of $\MfR$ with the manifold topology. 
\end{cor}
\begin{proof}
This follows from 
Theorem \ref{thm:qcrigidity}, 
by considering the Beltrami coefficient $\mu$
of the quasiconformal conjugacy $h$. Taking $h_t$ to be the quasiconformal 
map associated to $t\mu$ (normalized appropriately), we obtain
a family of maps $f_t=h^{-1}_t\circ f\circ h_t$ in $\SR$ that 
connects $f$ and $g$.
\end{proof}

\begin{remark} This implies Corollary~\ref{cor:conjconnected}.
\end{remark}

Finally, we require the fact that, within any given parameter
space $\MfR$, any parabolic point can be perturbed to an attracting one.

\begin{prop}[Perturbations of parabolic points]\label{prop:Shi}
Let $f\in\SR$, and suppose that $f$ has a parabolic periodic point
 $z_0$. Then there exists a function $f\in \MfR$, arbitrarily close to 
 $f$ in the manifold topology, such that $g$ has an attracting periodic point
 close to $z_0$.
\end{prop}
\begin{sketch}
 We could prove this using Epstein's transversality results for finite-type maps. Since these are currently unpublished, we shall instead sketch how to prove the proposition along the lines of Shishikura's argument in  \cite{MR892140}, 
see also \cite[Theorem 5]{alexmisha}.
 
Let $w\in\R$ be a point which belongs to the parabolic basin of $f$. Let $\eps>0$ and let
  $\phi$ be a real-quasiconformal map such that:
\begin{enumerate}
  \item $|\phi(z)-z|\leq \eps$ for all $z$;
  \item the complex dilatation of $\phi$ is bounded by $\eps$;
  \item $\phi$ fixes the orbit of $z_0$, as well as the points $\infty$ and $w$;
  \item $\phi$ is conformal on $\{z\in\C: |z-w|>\eps\}$;
  \item $\phi'(z_0)\in (0,1)$, and $\phi'=1$ on other points of the orbit of $z_0$.
\end{enumerate}
To define this map, note that outside a small neighbourhood of $w$, we can let
   $\phi(z)=z + \delta\cdot R(z)$, where $R$ is a real rational function  with a single pole at 
   $w$, having zeros along the orbit of $z_0$ and at $\infty$, and such that $R'(z_0)<0$ and $R'(z)=0$ at other points of the orbit of $z_0$. Provided $\delta>0$ 
   is sufficiently small, $\phi$ is a conformal diffeomorphism from $\{z\in\C: |z-w|>\eps\}$ onto its image, and
   one can extend $\phi$ to a neighbourhood of $w$ as a quasiconformal homeomorphism of the Riemann sphere $\overline{ \C}$ 
    with small dilatation. (Compare  \cite[Lemma 2]{MR892140}.)

Next consider the quasiregular map $F := \phi\circ f$. Then $z_0$ is a periodic point of
 period $n$ for $F$, and $(F^n)'(z_0) = (f^n)'(z_0)\cdot \phi'(z_0)$, so $z_0$ is an attracting
  periodic point. Moreover, let $U$ be a small attracting petal for $z_0$ as a periodic point of $f$,
   with $w\notin \overline{U}$.
   Provided that
   $\delta>0$ is sufficiently small, and that $R$ is chosen to have zeros of sufficiently high order along the     
   forward iterates of $z_0$ (except $z_0$), the set  $U$ will be contained in the basin of attraction of $z_0$ for $F$.
 
If $\eps>0$ is sufficiently small, then $\overline{f^n(B_{\eps}(w))}\subset U$ for some suitable $n$.
  Decreasing $\eps>0$ further, if necessary, the same will be true for $F$. Hence no orbit of $F$
  passes
  through $B_{\eps}(w)$ more than once, and we can apply quasiconformal surgery
to obtain the desired function $g$,   compare  \cite[Lemma 1]{MR892140}.
\end{sketch}

\subsection*{Combinatorial and analytic data} We now introduce data  that will
 allow us to encode when two real-topologically equivalent maps in
 $\SR$ are conformally conjugate (using Theorem
 \ref{thm:qcrigidity}). The notions of \emph{kneading sequences},
 which essentially determine combinatorial equivalence classes and of
 coordinates to ensure conformal conjugacy on attracting and
 parabolic basins are standard tools from the polynomial setting;
 we define and review them here briefly for completeness. To deal with
 escaping singular orbits, we will also require a new tool: 
 \emph{escaping coordinates}, which are provided by the results of
 \cite{boettcher}. 

\begin{defn}[Itineraries and kneading sequences] 
Let $f\in\SR$, and let $\mathcal{I}$ denote the set of connected components of
$\R\setminus\Crit(f)$. 
The \emph{itinerary} of a point 
$x\in\R$ is the sequence $\s= s_0 s_1 s_2 \dots$, where
$s_m = I_j$ if $I_j\in \mathcal{I}$ with
 $f^m(x) \in I_j$, or $s_m = f^m(x)$ if $f^m(x)$ is
 a critical point of $f$. 

Let $v_1<v_2<\dots < v_{k}$ be the
 singular values of $f$; the \emph{kneading sequence} of $f$ is the
 collection $(\s^1, \s^2, \dots, \s^k)$ of the itineraries of the
 $v_j$, together with the information which $v_j$ converge to
 an attracting cycle or to infinity.
\end{defn}

If two maps $f$ and $g$ are real-topologically equivalent, the map
$\phi$ allows us to relate the itineraries of $f$ and $g$.
Hence it makes sense to speak of two such maps having
{\lq}the same kneading sequence{\rq}. More formally:

\begin{defn}[The notion of having the same kneading sequences]
Let $f\in\SR$, let $v_1,\dots,v_k$ be the singular values of $f$ and
 let $\s^1,\dots,\s^k$ be their itineraries.

Let $g=\psi\circ f \circ \phi^{-1}$ be
 real-topologically
 equivalent to $f$.
 Then we say that $f$ and $g$ \emph{have the same kneading sequence}
 if
 \[ g^m(\psi(v_j)) \in \phi(s^j_m) \]
 for all $m\geq 0$ and $1\leq j \leq k$.
\end{defn}
\begin{remark}
Note that the definition depends on the maps $\psi$ and $\phi$, not
 only on the function $g$. We suppress this in the notation, which
 should not cause any confusion. 
\end{remark}

As stated above, our goal is to use kneading sequences to identify maps that 
 are 
 conformally conjugate; however, to do so we need to augment these with
 some analytic data. Indeed, for example the conformal conjugacy class
 of a map with attracting periodic orbits is not determined by the
 kneading sequence, since there will be invariant line fields on the basins
 of attraction. Similar issues are associated with parabolic orbits and
 escaping singular orbits. This can be dealt with in a 
 straightforward manner by introducing \emph{attracting, parabolic and
 escaping coordinates}.

 More precisely, let $f\in\SR$, and suppose that $f$ has an attracting
  periodic point $p\in\R$. Let $\Psi$ denote the linearizing coordinates 
  for $p$ (defined on the entire basin of the periodic attractor by the obvious 
  functional relation), normalized such that $\Psi'(p)=1$. Then the \emph{attracting
  coordinates for $f$ at $p$} consist of the multiplier $\mu$ of $p$ 
  together with the point
     \[ [\Psi(s_1):\Psi(s_2):\dots:\Psi(s_k)] \in \mathbb{CP}^{k-1}, \]
   where $s_1<s_2<\dots<s_k$ are the singular values of $f$ that are
   attracted by the cycle of $p$. The \emph{attracting coordinates for $f$}
   consists of the attracting coordinates at all attracting cycles of
   $f$, together with the information of which singular values are
   attracted to which attracting orbit. 

 If $f$ belongs to a real-analytic family $f_{\lambda}$, 
  say $f=f_{\lambda_0}$, 
  then for every $\lambda$ near $\lambda_0$
     there will be an attracting periodic point $p(\lambda)$ of
  $f_{\lambda}$ with $p(\lambda_0)=p$, depending real-analytically on
  $\lambda$. The linearizing coordinates for $p(\lambda)$ also depend
  analytically on $\lambda$, which implies that the corresponding
  attracting coordinates depend analytically on $\lambda$ 
  (provided we ignore any additional singular values of $f_{\lambda}$
  that may be
  attracted to $p(\lambda)$). 

 Similarly, one can define \emph{parabolic coordinates} at parabolic points, 
  which consist of the attracting Fatou coordinates of singular values,
  up to a translation. These will actually not be used in our proof
  of density of hyperbolicity, but we include them for completeness,
  to state the theorem below in full generality.

Finally, we also need to introduce analytic coordinates for
 singular values that are contained in the real part $\IR(f)$ of the
 escaping set. Such coordinates are given by \cite{boettcher}:

\begin{thm}[Escaping coordinates]
 Let $M$ be a real-analytic manifold with base point 
  $\lambda_0\in M$.
Also let $(f_\lambda)_{\lambda\in M}$ be a continuous family of functions
in $\SR$, all of which are real-topologically equivalent,
i.e. $\psi_\lambda\circ f_\lambda=f\circ \phi_\lambda$ with $\psi_\lambda$ and $\phi_\lambda$ 
depending continuously on $\lambda$. 

 Let $K\subset M$ and let $R$ be sufficiently large
 (depending on $K$). Then for every $\lambda\in K$, there exists
 a quasisymmetric map $h_{\lambda}:\R\to\R$ such that 
   \[ h_{\lambda}(f_{\lambda_0}(x)) = f_{\lambda}(h_{\lambda}(x)) \]
 whenever $x\in \IR(f_{\lambda_0})$ has the property that
  $f_{\lambda_0}^k(x)\geq R$ for all $k\geq 0$. 

 Furthermore, $h_{\lambda}(x)$ depends real-analytically on $\lambda$
  for fixed $x$. 
\end{thm} 

Using this map $h_\lambda$, we can now also define what it means that
two functions $f$ and $g\in\MfR$ have the \emph{same escaping
coordinates}: sufficiently large iterates of escaping singular 
values of $f$ should be carried to the corresponding iterates
for $g$ using this conjugacy on the escaping set.

Using these concepts, we can state the following result.

\begin{thm}[QC rigidity and conformal rigidity on the Fatou set] 
\label{thm:QC+con-rigidity}
Suppose that $f,g\in\SR$ are real-topologically equivalent.

Suppose also that $f$ and $g$ have the same kneading sequence and
 the same attracting, parabolic and escaping coordinates.

Then $f$ and $g$ are quasiconformally conjugate via
 a real-quasiconformal map $\phi:\C\to\C$ which is conformal on the 
 Fatou set of $f$.
\end{thm}
\begin{proof}
The assumption implies that $f$ and $g$ are combinatorially 
 conjugate and escaping conjugate, and that the combinatorial
 conjugacy can be chosen to be analytic in a neighbourhood of
 the part of the postsingular set that belongs to the Fatou set.
 Now it follows as in Theorem \ref{thm:realqcconj} that this conjugacy 
 promotes to a quasiconformal conjugacy, and this conjugacy is 
 conformal on the Fatou set.
\end{proof}

\section{Density of Hyperbolicity}
\label{sec:densityhyperbolicity}

The basic idea in our proof of density of hyperbolicity is to create 
 more and more critical relations of a suitable type near a starting
 parameter, and restrict to a submanifold where this critical relation
 is persistent. To make this work, we need to know that we can carry out
 this process in such a way that 
 the dimension of the manifold is not reduced by more than one would expect.
 We could do so by applying deep transversality results due to Adam Epstein
 (though it is not entirely clear how to apply these e.g.\ to work
 with escaping coordinates). Instead, we use the following, much softer
 statement.
 
\begin{thm}[Finding submanifolds] \label{thm:lojasiewicz}
Let $U\subset\R^n$ be an open ball, and let 
  \[ \rho: U \to \R \]
be real-analytic. Suppose that $x_1,x_2\in U$ satisfy
 $\rho(x_1)\neq \rho(x_2)$, and let $\nu\in\R$ be a value between 
 $\rho(x_1)$ and $\rho(x_2)$.

Then there exists $w\in U$ with $\rho(w) = \nu$ such that
 $\rho^{-1}(\nu)$ is a real-analytic $(n-1)$-dimensional manifold near $w$. 
\end{thm}
\begin{proof}
 By continuity of $\rho$, the set $A := \rho^{-1}(\nu)$ separates $U$. 
  This means that $A$ has topological dimension at least $n-1$     \cite{dimensiontheory}.

Now the zero set of a real analytic function is a subanalytic, and indeed
 semianalytic, set.
Subanalytic sets can be written as a locally finite 
union of real-analytic submanifolds, see 
\cite{MR972342}. So $A$ contains a real-analytic   manifold
 of the same dimension as its topological dimension.
 Compare also Lojasiewicz's structure theorem
  for real-analytic sets \cite[Theorem 6.3.3]{MR1916029}.
\end{proof}

\begin{prop}[Abundance of critical relations] \label{prop:density1}
 Let $f\in\SR$ and let 
   $M$ be a real-analytic submanifold of $\MfR$ of dimension $n$.
  Suppose that no two maps $f,g\in M$ have the same kneading sequence. 
 Then, given any $f_0\in M$, there exists some $f\in M$,
 arbitrarily close to $f_0$, such that $f$ satisfies
 $n$ non-persistent critical relations of the form
 $f(v)=c$, where $v$ is a singular value and $c$ is   a real critical point.
\end{prop}
\begin{remark}
 Strictly speaking, our statement is ambiguous, since ``having the same kneading sequence''
  depends on the choice of $\phi$ and $\psi$ from the definition of real-topological
  equivalence. Since our conclusion is local, 
  the assumption should also be understood locally:
  for $f_0\in M$ we can find a neighbourhood $U\subset M$ on which the maps 
  $\phi$ and $\psi$ can be chosen to depend continuously, and no two maps in $U$
  should have the same kneading sequence with respect to this choice.
\end{remark}
\begin{proof}
We will prove the theorem by induction on $n$. If $n=0$, then there is
 nothing to prove. So suppose the theorem holds when the dimension of 
 $M$ is $n-1$. Now assume $M$ has dimension $n$. Let 
 $f_0\in M$; by a small perturbation,
 we may assume that $f_0$ does not satisfy
 any non-persistent critical relation of the form $f(v)=c$. Let $U$ be
 a small neighbourhood of $f_0$ as in the above remark; in particular,
 the critical and asymptotic values of $f\in U$ depend continuously
 and real-analytically (with the respect to the manifold structure).
 We can choose $U$ to be real-analytically diffeomorphic to an open ball
 in $\R^n$. 

Since no two maps in $U$ have the same kneading sequence,
 there must be maps $f_1,f_2\in U$, as well as a critical point $c=c(f)$ and a critical value $v(f)$ such that 
 $f_1^n(v(f_1)) - c(f_1)$ and $f_2^n(v(f_2)) - c(f_2)$ have 
 opposite signs for some $n\in\N$.

Set 
   \[ \rho(f) := f^n(v(f)) - c(f); \]
 then $\rho$ is real-analytic, and we can apply Theorem \ref{thm:lojasiewicz}
 to $U$, $\rho$, $x_1:=f_1$, $x_2 := f_2$ and  $\nu = 0$.
 We obtain an $(n-1)$-dimensional analytic submanifold $N$ of $M$, contained
 in $U$, such that all maps $f\in N$
  satisfy $\rho(f)=0$; i.e., they satisfy a 
  critical relation of the desired form which is non-persistent in $M$,
  but persistent in $N$. 

Applying the induction hypothesis, we find a map $f\in N$
 that satisfies $n-1$ critical relations which are non-persistent in $N$,
 and hence $n$ critical relations which are non-persistent in $M$, as
 desired. 
\end{proof}

 Recall that $L_I(f)$ is the set of points $z\in J(f)\setminus (J_r(f)\cup I(f))$
whose orbits accumulate on escaping singular  orbits under iteration.
The following theorem shows density of hyperbolicity
provided this set does not support invariant line fields.

\begin{thm}[Density of hyperbolicity when $L_I$ has no invariant line fields]
Let $f\in\SR$. Let $U\subset\MfRt$ be open with the property that no
  $g\in U$ has an invariant line field on the set
  $L_I(g)$.

 Then $U$ contains a real-hyperbolic function.
\end{thm}

\begin{proof}
Let $f_1\in U$ be such that the number of singular values that
belong to attracting basins is locally maximal near $f_1$. Then
there is an open neighbourhood $U'\subset U$ of $f_1$ such that
all $g\in U'$ have the same number, say $k_1$, of such singular values.
By Proposition~\ref{prop:Shi},
 this implies that no function in $U'$ has any parabolic
periodic points.

Now, similarly, pick $f_2\in U'$ such that the number $k_2$ of singular
values that tend to infinity under iteration is locally maximal, 
and let $U''$ be an open neighbourhood of $f_2$ such that all
maps in $U''$ have $k_2$ such singular values. 

Set $q:= \# S(f)$; recall that $\MfRt$, and hence $U''$ has dimension
  $q$. 
  We may assume that $U''$ is chosen sufficiently small that
  there is a real-analytic section $U''\to \MfR$. (In fact, if $f$ is
  not periodic, then there is
  even a global section $\MfRt\to\MfR$, see Appendix \ref{app:parameter}.) 
  So we can identify $U''$ with a $q$-dimensional submanifold of $\MfR$
  in which no two maps are conformally conjugate. 

Applying Theorem~\ref{thm:lojasiewicz} repeatedly, we
 can find a manifold $M\subset U''$ of dimension
 $q' := q - k_1 - k_2$ on which the attracting and escaping
 coordinates are constant. 

By Theorem~\ref{thm:QC+con-rigidity}, 
any two maps in $M$  that have the same kneading sequence
 would be quasiconformally conjugate, and the dilatation 
 would be supported on the Julia set. However, by the assumption, 
Lemma~\ref{lem:decomp} and Theorems~\ref{thm:nolinefieldJr} and
\ref{thm:nolinefieldsLI},
  this means that the maps would be
 conformally conjugate, and hence equal. 
 Thus no two maps in $M$ have the same kneading
 sequence.  Note that within $M$ the 
$k_1+k_2$ singular values in attracting basins or tending to
infinity do not satisfy any non-persistent relations.
(Two singular values are said to have a relation if they have  the same grand orbit.)

Now we apply Proposition~\ref{prop:density1}, to obtain a function
 $g\in M$ that satisfies $q'$ non-persistent critical relations of the
 form $g(v)=c$, where $v$ is a singular value and $c$ is 
 a real critical point. This means that every singular
 value is eventually either mapped to a superattracting cycle or to
 one of the $k_1+k_2$ singular values that belong to
 attracting basins or to $I_{\R}(g)$ (and which, by assumption, do
 not satisfy any singular relations). Thus all singular values of
 $g$ belong to attracting basins or converge to infinity, as claimed.
\end{proof}

\begin{cor}[Density of hyperbolicity for bounded functions]\
\label{cor:denboundedfunction}
 Let $f\in\SR$ and suppose that one of the following holds:
  \begin{enumerate}
     \item $f|\R$ is bounded;
     \item $f$ satisfies the
  sector condition; or
    \item $f$ has bounded criticality.
   \end{enumerate}
  Then if (a) holds then hyperbolicity  is dense in
  $\MfRt$.
  If (b) or (c) hold then  real-hyperbolicity is dense in   $\MfRt$.
\end{cor}
\begin{proof}
 This is an immediate consequence of the previous result. (Note that each of the conditions 
 (a), (b) and (c) 
    is invariant under real-quasiconformal
    equivalence, and hence 
    no function $g\in \MfR$ supports an invariant line field on $L_I(g)$ by Theorems
\ref{thm:nolinefieldJr} and \ref{thm:nolinefieldsLI}.)
    \end{proof}
 
This proves Theorem \ref{thm:perturbbounded} (and the corresponding statement
 for maps satisfying the sector condition). We now deduce Theorem
 \ref{thm:densityunbounded} (of which Theorem \ref{thm:densitybounded} is a special case). 

\begin{proof}[Proof of Theorem \ref{thm:densityunbounded}]
 Real-hyperbolicity is an open property, so we need only prove that
  real-hyperbolicity is dense. Let $\lambda_0\in N$ and set $f:= f_{\lambda_0}$. 
  We may assume (perturbing $\lambda_0$ 
  if necessary) that the number $m := \# S(f)$ 
  is locally maximal. For $\lambda$
  sufficiently close to $\lambda_0$, $f_{\lambda}$ must have at least $m$ singular 
  values (see Lemma \ref{lem:singulardependence}). So we may assume, by shrinking $N$ 
  if necessary, that $\# S(f_{\lambda})= m$ for all $\lambda\in N$. 

 By Proposition \ref{prop:continuity}, all $f_{\lambda}$ belong to $\MfR$, and 
  $f_{\lambda}$ is a continuous family with respect to the topology of $\MfR$. Via the natural
  projection $\MfR\to \MfRt$, we obtain a continuous map
    \[ \Phi : N \to \MfRt \]
  such that $f_{\lambda}$ and (any representative of) $\Phi(f_{\lambda})$ are conformally conjugate. In particular,
  the map $\Phi$ is injective. Since $m\leq n$, we must in fact have
  $m=n$ and $\Phi$ is locally surjective 
  (by the Invariance of Domain Theorem). 
  The claim now follows
  from the preceding theorem. 
\end{proof}

\begin{proof}[Proof of Corollary~\ref{cor:thm1point7}]
This follows from Corollary~\ref{cor:denboundedfunction}
and the fact that the number of singularities is preserved under
topological equivalence.
\end{proof}

\section{Circle maps}
\label{sec:circlemaps}%
The adaptation of our results to circle maps is straightforward, and 
 essentially in complete analogy with the case of bounded functions
 $f\in\SR$. To formulate
 results for transcendental maps and Blaschke products at the same time,
 let us denote by $X$ the union of $\SC$ with the set of all
 Blaschke products of degree at least two which preserve
 $\{0,\infty\}$ and for which all critical values are
 contained in $S^1\cup \{0,\infty\}$.

\begin{lem}[Circle maps without critical points] \label{lem:circlemapcrit}
If $f\in X$ has no critical points in $S^1$, then 
$f(z)=z^d$ with $d\ne 0$.
\end{lem}
\begin{proof}
Let $Z=f^{-1}(S^1\cup \{0,\infty\})$. Since the critical values of $f$ are on
the circle but $S^1$ contains no critical points, $S^1$ is one of the connected components
of $Z$. We claim that it is the only nontrivial component (i.e.\ consisting of more than one point) of $Z$. Indeed, otherwise there is at least one multiply-connected component $V$ of $\C\setminus Z$ that is not a punctured disk. But $f|_{V}$ is a covering
whose image is either $\D\setminus \{0\}$ or $\C\setminus \cl{D}$, which is impossible. 
It follows that $f$ has no singular values in $\C^*$, and hence $f(z)=z^d$ for some $d\in\Z\setminus\{0\}$.
\end{proof}
\begin{thm}[QC rigidity for circle maps]\label{thm:qsrigidcircle}
 Suppose that two maps $f,g\in X$ are $S^1$-topologically
  equivalent and combinatorially
  conjugate. Then the two maps are $S^1$-quasiconformally conjugate via a 
  conjugacy that agrees with the combinatorial conjugacy on the 
  postsingular set.
\end{thm}
(Here the notions of $S^1$-topological equivalence as well as combinatorial and
   $S^1$-qc conjugacy are defined in analogy to the real case. Note that by definition 
   a combinatorial conjugacy sends parabolic to parabolic points.)  
\begin{proof}
  Note that from the previous lemma $f$ and $g$ have at least one critical point
  (unless $f,g$ are of the form $z\mapsto z^d$). 
   It follows
   from \cite{CS} that the two maps are quasisymmetrically conjugate on the
   circle. 
 Applying a pullback argument yields a quasiconformal conjugacy
  on the entire complex plane. 
\end{proof}

The second ingredient is the absence of invariant line fields theorem:
\begin{thm}[No invariant line fields]\label{thm:nolinefieldcircle}
 A map $f\in X$ does not support any invariant line fields on its Julia set.
\end{thm}

Here, once again, the absence of invariant line fields on the set of points with
  bounded orbits follows from \cite{CS}. The absence of
 invariant line fields on the radial Julia set does not follow directly from
  \cite{linefields}, since the function $f$ is not necessarily 
  meromorphic in the
  plane but can be proved in the same manner. Alternatively, the result
  of \cite{linefields} is generalized in \cite{conicalrigidity} to 
  arbitrary \emph{Ahlfors islands maps}, and this result can be applied 
  directly to $f$. 

Finally, for $f\in\SC$, absence of invariant line fields on the ``escaping set''
\[ I(f) := \{z\in\C: \omega(z)\subset \{0,\infty\}\} \]
 follows from the following theorem, which is proved completely analogously
 to the corresponding result from \cite{boettcher}.

\begin{thm}[No invariant line fields on the escaping set]\label{thm:nolinefieldifcircle}
 Let $f:\C^*\to\C^*$ be a transcendental self-map of the punctured plane, with essential
  singularities at $0$ and $\infty$. Suppose that the set
  $S(f)\setminus\{0,\infty\}$ is compactly contained in $\C^*$.

 Then the set $I(f)$ does not support invariant line fields.
\end{thm}

Again, analogously to the case of $\SR$, the set
 $\MfC$ of functions $S^1$-topologically equivalent to $f$ has the
 structure of a real-analytic manifold of dimension $q+1$, where $q=\# S(f)$. 
 As we saw in Remark 1 below
 Lemma~\ref{lem:trigoaffinerigid}, its quotient
 $\MfCt$ by conjugation by rotations is not a manifold.
 However, it is a $q$-dimensional orbifold. 
 We then obtain by the same proof as for functions
 in $\SR$:

\begin{thm}[Density of hyperbolicity for circle maps]
 Let $f\in X$. Then hyperbolicity is dense in $\MfCt$.
\end{thm}

 The theorems on circle maps stated in the introduction follow
 from the preceding result in the same manner as for real entire
 functions:

\begin{proof}[Proof of Theorem~\ref{thm:densitycircle}]
 This follows from the preceding theorem in the same
manner as for real entire function in $\SR$. 
\end{proof}

\begin{proof}[Proof of Theorem~\ref{thm:qcrigcircle}]
This  is a special case of Theorems~\ref{thm:qsrigidcircle} and \ref{thm:nolinefieldcircle}.
\end{proof}

\begin{proof}[Proof of Corollary~\ref{cor:trigopol}] The first statement 
of this corollary follows Theorem~\ref{thm:nolinefieldifcircle}.
Part (a) follows   from  Theorems  \ref{thm:qsrigidcircle} and
Lemma~\ref{lem:trigoaffinerigid} using the same argument as in
the proof of Corollary \ref{cor:conjconnected2}.
Part (b) follows from Theorems \ref{thm:nolinefieldcircle} and \ref{thm:nolinefieldifcircle}.
\end{proof}

\appendix

\section{More on parameter spaces} \label{app:parameter}

\begin{prop}[Real-analytic structure of parameter spaces]\label{prop:manifoldstructure}
 Let $f\in\SR$ and set $q:= \# S(f)$. 
  If $f$ is not periodic (i.e., there is no $\kappa\in\R\setminus\{0\}$ with
  $f(x+\kappa)=f(x)$ for all $x$), then
    \[ \MfR \raequiv \R^{q+2}\quad \text{and}\quad
       \MfRt \raequiv \R^q \]
  (where $\raequiv$ denotes real-analytic isomorphism). Otherwise, 
    \[ \MfR \raequiv \R^q\times S^1\times S^1 \quad\text{and}\quad
       \MfRt \raequiv \R^{q-1}\times S^1. \]

 More precisely, let us set 
 \[    \Lambda := \{(a_1,\dots,a_q)\in\R^q: a_1<a_2<\dots<a_q \}. \]
  Then there exists a real-analytic covering map
     \[ \Theta: \Lambda \times (0,\infty)\times \R \to \MfR \]
  with the following properties.
 \begin{enumerate}
  \item If $\lambda= (a_1,\dots,a_q)\in \Lambda$, $a>0$ and $b\in\R$, then
    the singular values of $g := \Theta(\lambda,a,b)$ are exactly $a_1,\dots, a_q$.
    Furthermore, if $f$ is periodic, then $g$ is also periodic of
    minimal period $a\cdot \kappa$, where $\kappa$ is the minimal period of $f$.
  \item Let $\lambda_0 = (s_1,\dots,s_q)$ be the singular values of $f$. Then
    $f = \Theta(\lambda_0 , 1 , 0)$.
  \item If $f$ is not periodic,  $\Theta$ is a diffeomorphism. Otherwise,
    $\Theta(\lambda,a,b)=\Theta(\lambda',a',b')$ if and only if 
    $\lambda=\lambda'$, $a=a'$ and $b$ and $b'$ differ by a multiple of $a\cdot\kappa$.
  \item Fix $a>0$ and $b\in\R$. If $f$ is not periodic, then $\Theta(\lambda,a,b)$ is not 
    conformally 
    conjugate (via a map from $\MobR$)
    to $\Theta(\lambda',a,b)$ for $\lambda\neq \lambda'$. Otherwise, these maps
    are conjugate if and only if there is $m\in\Z$ such that
    $\lambda'$ is obtained from $\lambda$ by adding $m\cdot a \cdot \kappa$ to all entries. 
 \end{enumerate}
\end{prop}
\begin{remark} A somewhat related theorem appears in \cite[Theorem 4.1]{demelovanstrien}
where it is shown that one can parametrise the space of real polynomials  (as in Theorem~\ref{thm:realgidityvS} anchored at $0$ and $1$) with $d$ distinct critical points $c_1,\dots,c_d$ -- all of which are assumed to be real -- by its critical values $v_1,\dots,v_d$. The difference is that in that theorem we allow critical values
to coincide.
\end{remark}
\begin{proof}
 The idea is to start with a family of quasiconformal functions $\psi_{\lambda}\in\HomeoR$, 
  $\lambda\in\Lambda$, where
   $\psi_{\lambda}$ takes $\lambda_0$ to $\lambda$, and then solve the Beltrami equation 
   to obtain $\phi_{\lambda}$ such that $\psi_{\lambda}\circ f \circ \phi_{\lambda}^{-1}$ is
   an entire function. There is a choice of normalization of $\phi_\lambda$, which gives rise
    to 
   the additional two real parameters come from.

  If there are at least two real preimages of singular values, then it is easy to
   obtain a natural normalization of $\phi_{\lambda}$. In order to obtain a 
   construction that works in all cases, we will proceed in a slightly more ad-hoc
   manner. 

 If $f$ is not periodic, let us set $\kappa:=1$,
  otherwise $\kappa$ is the minimal period of $f$ as defined above.

We define a
 real-analytic family $\psi_{\lambda}:\C\to\C$ 
 with $\psi_{\lambda}(s_j)=a_j$, where
 $\lambda=(a_1,\dots,a_q)\in\Lambda$. If
 $a_1=s_1$ and $a_q=s_q$, let $h:\R\to\R$ be 
 the unique map with $h(s_j)=a_j$ that is linear
 on every component of $\R\setminus S(f)$ and asymptotic to the identity
 at $\infty$. We define
 $\psi_{\lambda}(x+iy) := h(x) + iy$. In particular,
 $\psi_{\lambda_0}=\id$. 

Otherwise, set 
  \[ A(z) := (z-s_1)\cdot \frac{a_q-a_1}{s_q-s_1} + a_1 \]
 and $\tilde{\lambda} := (A^{-1}(a_1),\dots,A^{-1}(a_q))$. We define
 $\psi_{\lambda}(z) :=
  A(\psi_{\tilde{\lambda}}(z))$.

By construction, the family $\psi_{\lambda}$ has the following property:
 \begin{center}
  Let $\lambda_1=(a_1,\dots,a_q)\in\Lambda$ and
   $A(z)=az+b$ be a real-affine map. If we set
   $\lambda_2 := (A(a_1),\dots,A(a_q))$, then 
   $\psi_{\lambda_2}\circ \psi_{\lambda_1}^{-1} = A$.
 \end{center}

Let $\mu_{\lambda}$ be the complex dilatation of $\psi_{\lambda}$. 
 We can pull back under $f$ to obtain
 a complex structure $\nu_{\lambda} := f^*(\mu_{\lambda})$. 
 By the Measurable Riemann Mapping theorem, see for example \cite{MR2241787}, we can find
 a quasiconformal homeomorphism $\phi_{\lambda,a,b}:\C\to\C$ whose complex
 dilatation is given by $\nu_{\lambda}$. This map is uniquely determined
 if we require that $\phi_{\lambda,a,b}(0)=a\cdot b$ and 
 $\phi_{\lambda,a,b}(\kappa)=a\cdot(b+\kappa)$.

The functions $\phi_{\lambda,a,b}$ depend real-analytically on $\nu_{\lambda}$, 
 and hence on $\lambda$, as well as on $a$ and $b$. 
  By a well-known argument, see for example
 \cite[Page 21]{buffcheritat}, the family
   \[ \Phi(\lambda,a,b) := f_{\lambda,a,b} := \psi_{\lambda}^{-1}\circ f \circ  \phi_{\lambda,a,b} \]
 also depends analytically on $(\lambda,a,b)$.
Clearly we have $f_{\lambda_0,1,0}=f$ and 
 $S(f_{(a_1,\dots,a_q),a,b})=\{a_1,\dots,a_q\}$.

Since $f$ is real, the Beltrami differential $\nu_{\lambda}$ is
 symmetric with respect to the real axis (i.e.\ 
 $\nu_{\lambda}(\bar{z})=\overline{\nu_{\lambda}(z)}$, and hence
 the normalization ensures that $\phi_{\lambda,a,b}$ restricts to an
 orientation-preserving homeomorphism of the real line. Thus
 $f_{\lambda,a,b}\in\MfR$ for all $\lambda$. 

Similarly, if $f$ is periodic, then $\nu_{\lambda}$ is periodic
 with period $\kappa$, and the normalization ensures that
 $\phi_{\lambda,a,b}(z+\kappa)=\phi_{\lambda,a,b}(z)+a\kappa$ for all $z$.
 Thus each $f_{\lambda,a,b}$ is periodic with period $b\kappa$. We can
 apply the same argument to see that $b\kappa$ is the minimal period of
 $f_{\lambda,a,b}$. Indeed, we write
 $f = \psi_{\lambda}\circ f_{\lambda,a,b}\circ \phi_{\lambda,a,b}^{-1}$.
 If $b\kappa' \leq b\kappa$ is a period of $f_{\lambda,a,b}$, then we see
 as above that $\phi_{\lambda,a,b}^{-1}(z+b\kappa')=\phi_{\lambda,a,b}^{-1}(z)+c$ for
 some $c>0$. Clearly $c\leq \kappa$, and by construction $c$ is
 a period of $f$. Thus $c=\kappa=\kappa'$, as claimed. 

The remaining claims follow from the construction. Indeed,
 suppose that $f_{\lambda,a,b}=f_{\lambda',a',b'}$. Then $\lambda=\lambda'$
 (because these are the singular values) and $a=a'$ (because this is the period).
 By construction, we have $f_{\lambda,a,b'}(z)=f_{\lambda,a,b}(z+ab-ab')$. Hence
  $a(b-b')$ is a period of $f_{\lambda,a,b}$, and hence $b-b'$ is a multiple
  of $\kappa$ (if $f$ is periodic) or $b=b'$ (otherwise). 

 Now fix $a$ and $b$ and 
  suppose now that $f_{\lambda} := f_{\lambda,a,b}$ and $f_{\lambda'} := f_{\lambda',a,b}$ 
 are conformally conjugate
 by some real-affine map $A(z)=\alpha z+\beta$, $\alpha>0$, $\beta\in\R$. Then it follows from
 the property stated above that
 $\psi_{\lambda'}\circ\psi_{\lambda}^{-1} = A$, and hence
 $A\circ f_{\lambda}\circ A^{-1} = f_{\lambda'} = A\circ f_{\lambda'}$. 
 In particular, we must have $\alpha=1$ and $\beta$ is a period of $f_{\lambda}$;
 i.e., $f$ is periodic and $\beta$ is an integer multiple of $a\cdot \kappa$.
\end{proof}

\begin{lem}[Dependence of singular values] \label{lem:singulardependence}
 Let $f:\C\to\C$ be an entire function, and let $f_n$ be entire functions with
  $f_n\to f$ locally uniformly. If $a\in S(f)$, then for sufficiently large $n$,
  there is $a_n\in S(f_n)$ such that $a_n\to a$.
\end{lem}
\begin{proof}
 See e.g.\ \cite{MR1477036}.
\end{proof}

\begin{prop}[All continuous families arise from QC equivalence]\label{prop:appendixfamilies}
 Let $n,k\in\N$ and suppose that $(f_t)_{t\in[-1,1]^k}$ is a continuous family of functions
  $f_t\in \SR$ such that $\# S(f_t)=n$ for all $t$. Then there exist
  continuous families $\phi_t,\psi_t\in\HomeoR$ such that
    \[ f_t = \psi_t\circ f_0\circ \phi_t^{-1} \]
  for all $t\in [-1,1]$.
\end{prop}
\begin{sketch}
  We first note that the assumption implies that the
  singular values of $f_t$ move continuously by Lemma \ref{lem:singulardependence}. 
  That is, there 
  are continuous functions $s_1,\dots,s_n:[-1,1]^k\to \R$ with
   $s_1(t)<s_2(t)<\dots < s_n(t)$ for all $t$ and
   $S(f_t)=\{s_1(t),\dots,s_n(t)\}$. We set $s_j := s_j(0)$ (where $0$ denotes the origin in $[-1,1]^k$).

  Choose a continuous family $\psi_t$ of real-quasiconformal homeomorphisms 
   such that $\psi_t(s_j)=s_j(t)$ for all $t\in [-1,1]$ and $j\in\{1,\dots,n\}$, such that
   $\psi_0=\id$, and such that 
   $\psi_t(z)=z$ whenever $|\im z|\geq 1$. 
   (For example, we can define $\psi_t$ in a piecewise linear manner.) 

   By solving the Beltrami equation (similarly as above), we can also find a continuous
   family $\phi_t$ of real-quasiconformal homeomorphisms such that
   \begin{equation}\label{eqn:gt} g_t := \psi_t^{-1}\circ f_t \circ \phi_t \end{equation}
   is an entire function for every $t$. The function $\phi_t$ is uniquely determined up to 
   pre-composition by an element of $\MobR$. 

 The idea is to normalize $\phi_t$ in such a way as to ensure that $g_t$ agrees with $f_0$ at a chosen base point, and conclude that
  $g_t=f_0$ for all $t$. It is not a priori clear that such a normalization can always be carried out globally (that this is true shall follow
   \emph{a posteriori} from the proof), but locally this follows from the argument principle. 
   Hence the following claim (which requires a specific normalisation 
   for the quasiconformal maps $\phi_t$) 
   provides a local version of our proposition:

\begin{claim}
  Let $D\subset [-1,1]^k$ be a connected subset with $0\in D$, and suppose that there is a 
    continuous function $t\mapsto \zeta_t\in \C$, defined on $D$, such that $f_t(\zeta_t)=i$ 
  for all $t\in D$   (where $i$, as usual, denotes the imaginary unit). 
    Define $(g_t)_{t\in D}$ by~\eqref{eqn:gt}, where $\phi_t$ is normalized such that
    $\phi_t(\zeta_0)=\zeta_t$. Then $g_t=f_0$ for all $t\in D$.
\end{claim}   
\begin{subproof} We use the concept of 
    \emph{line complexes} from classical function theory. Fix $n+1$ pairwise disjoint 
    arcs 
    $\gamma_0,\dots,\gamma_n$ connecting $i$ and $-i$ and each intersecting the real axis in
    precisely one point, in such a way that different arcs intersect at $\pm i$, and such that the arcs intersect the real axis in different
    intervals of $\R\setminus\{s_1,\dots,s_n\}$. We can assume that the arcs are ordered such that their intersection points with $\R$
    are listed in increasing order; then
    $\gamma_j\cup \gamma_{j-1}$ is a Jordan curve separating $s_j$ from $\infty$ and all other
    $s_{j'}$. The \emph{line complex} $LC(g_t)$ is the preimage of
    $\bigcup \gamma_j$ under $g_t$.

   More precisely, we can think of $LC(g_t)$ as an abstract graph with a base point
    and colored edges. The vertices 
    are the elements of the set $g_t^{-1}(\{i,-i\})$, and the base point is the vertex
    represented by $\zeta_0$. Two vertices $z_1$ and $z_2$ are connected by an edge of color
    $j\in \{0,\dots,n\}$ if and only if there is a component of
    $g_t^{-1}(\gamma_j)$ that connects $z_1$ and $z_2$. The following two facts are
    classical:
   \begin{itemize}
    \item The line complex $LC(g_t)$ depends continuously on $t$ as a graph. 
      (By this we mean that, for any fixed $N$, the part of $LC(g_t)$ within distance
       at most $N$ of $\zeta_0$ is locally constant.) Hence, since $[-1,1]$ is connected,
       it follows that all the abstract graphs $LC(g_t)$ are isomorphic.
    \item With the above normalization, the function $g_t$ is uniquely determined by
      its line complex $LC(g_t)$.
   \end{itemize}
  The first of these is elementary: It follows from the fact that the analytic continuation
   of $f_t^{-1}$ along a 
   fixed composition of the curves $\gamma_j$ will depend continuously on
   $t$. To reconstruct the function $g_t$ from its line complex, we need only 
   build the Riemann surface of $g_t^{-1}$ by pasting together copies of the upper and
   lower half plane as specified by the line complex. The resulting entire function is 
   determined uniquely up to precomposition by a map in $\MobR$; in other words, the function
   is fixed once we require that the base point of its line-complex be placed at $\zeta_0$. For details, compare
   \cite{MR2435270}.
 \end{subproof}
 We observe that, in this article, we only ever  use a local version of Proposition \ref{prop:appendixfamilies} that
   asserts the existence of the maps $\psi_t$ and $\phi_t$ in a neighbourhood of the origin; this version follows immediately from the Claim. 

 To also deduce the global
  statement, note that we can apply our claim near any given base point $t_0\in [-1,1]^k$. This implies that the full set of preimages
  $f_t^{-1}(i)$ moves continuously near every point $t_0$. Since $[-1,1]^k$ is simply-connected, it follows that this set moves
   continuously thorughout the entire family. In particular, we can take $D=[-1,1]^k$ in the Claim, and are done.
\end{sketch}

\small

\newcommand{\noopsort}[1]{}\providecommand{\href}[2]{#2}\def\polhk#1{\setbox0=\hbox{#1}{\ooalign{\hidewidth
  \lower1.5ex\hbox{`}\hidewidth\crcr\unhbox0}}}
  \def\polhk#1{\setbox0=\hbox{#1}{\ooalign{\hidewidth\lower1.5ex\hbox{`}\hidewidth\crcr\unhbox0}}}
  \input{cyracc.def} \def\j{{\u i}} \def\J{{\u I}} \newfont{\cyrit}{wncyi10 at
  12pt}\def\cprime{$'$}
\providecommand{\bysame}{\leavevmode\hbox to3em{\hrulefill}\thinspace}
\providecommand{\href}[2]{#2}

\end{document}